\documentclass[final]{siamltex} 
\usepackage{amsmath,amssymb}
\usepackage{color}
\usepackage{enumerate}
\usepackage{multirow}
\usepackage{booktabs}
\usepackage{array}
\usepackage{pstricks}
\usepackage{euscript}
\usepackage{epic,eepic}
\usepackage{graphicx}

\newcommand{\menge}[2]{\big\{{#1} \mid {#2}\big\}}

\newcommand{\emp}{\ensuremath{{\varnothing}}}

\newcommand{\scal}[2]{\left\langle{#1},{#2} \right\rangle}

\newcommand{\HH}{\ensuremath{\mathcal H}}

\newcommand{\EE}{\ensuremath{\boldsymbol{E}}}
\newcommand{\E}{\ensuremath{\mathbb{E}}}

\newcommand{\RR}{\ensuremath{\mathbb R}}

\newcommand{\RP}{\ensuremath{\left[0,+\infty\right[}}
\newcommand{\RPP}{\ensuremath{\,\left]0,+\infty\right[}}

\newcommand{\NN}{\ensuremath{\mathbb N}}

\newcommand{\dom}{\ensuremath{\operatorname{dom}}}

\newcommand{\prox}{\ensuremath{\operatorname{prox}}}

\newcommand{\Argmin}{\ensuremath{\operatorname{Argmin}}}
\newcommand{\argmin}{\ensuremath{\operatorname{argmin}}}

\newcommand{\zz}{\ensuremath{\boldsymbol{z}}}

\newcommand{\EEE}{\ensuremath{\boldsymbol{E}}}

\newcommand{\weakly}{\ensuremath{\rightharpoonup}}

\newcommand{\pinf}{\ensuremath{+\infty}}

\newcommand{\obj}{T}
\newcommand{\smo}{L}
\newcommand{\nsm}{R}
\newcommand{\X}{\mathcal{X}}
\newcommand{\Y}{\mathcal{Y}}

\PassOptionsToPackage{normalem}{ulem}

\newtheorem{condition}[theorem]{Condition}
\newtheorem{problem}[theorem]{Problem}
\newtheorem{example}{Example}
\newtheorem{remark}{Remark}
\newtheorem{algo}[theorem]{Algorithm}
\newtheorem{assumption}[theorem]{Assumption}

\title{Convergence of Stochastic Proximal Gradient  Algorithm}
 \author{Lorenzo Rosasco  $^\dag$\thanks{
       DIBRIS, Universit\`a di Genova,
       Via Dodecaneso, 35,
       16146, Genova, Italy, ({\tt lrosasco@mit.edu})}   
 \and
Silvia Villa $^\dag$ \and B$\grave{\text{\u{a}}}$ng C\^ong V\~u  
\thanks{
 LCSL, Istituto Italiano di Tecnologia
       and Massachusetts Institute of Technology,
       Bldg. 46-5155, 77 Massachusetts Avenue, Cambridge, MA 02139, USA, ({\tt Silvia.Villa@iit.it, Cong.Bang@iit.it})} 
}

\begin{document}
\maketitle
\frenchspacing

\begin{abstract}
We prove novel  convergence results for a stochastic proximal gradient algorithm 
suitable for solving a large class of convex optimization problems, where a convex objective function  is given by the sum of a smooth and a possibly non-smooth component.
We consider the iterates convergence and   derive $O(1/n)$ non asymptotic bounds in expectation in the strongly convex case, as well as  almost sure convergence results  under weaker assumptions. Our approach allows to avoid averaging and  weaken boundedness assumptions which are often considered in theoretical studies and might not be satisfied in practice. 
\end{abstract}


\maketitle

\section{Introduction}

First order methods  have recently been widely applied 
to solve convex optimization problems in a variety of areas including 
machine learning  and signal processing. 
In particular, proximal
gradient algorithms (a.k.a. forward-backward splitting algorithms) and their accelerated variants have received considerable attention (see \cite{livre1,siam05,nesterov07,beck09} 
and references therein). 
These algorithms are easy to implement and suitable for solving high dimensional  problems 
 thanks to the low  memory requirement of each iteration. Moreover, they are particularly suitable for composite optimization, that is 
 when a convex objective function  is  the sum  of  a smooth and a non-smooth component.  
This class of optimization problems arises naturally in  regularization schemes where one component is a data fitting term  and the other a regularizer, see for example \cite{siam07,MRSVV10}. 
Interestingly,    proximal  splitting algorithms separate the contribution of each   component  at every iteration:
 the proximal operator defined by the non smooth term is applied to a  gradient descent step for  the smooth term.  In practice it is often relevant to consider situations where the latter operation cannot be perfomed {\em exactly}. For example the case where the proximal operator is known only up-to an error have been considered  in \cite{Roc76,siam05,SRB11,VilSal13}. 

In this paper we are interested in the complementary situation where it is the gradient 
of the smooth term to be know up-to an error. More precisely, we consider the case where 
only  stochastic estimates of the gradient are available and develop stochastic versions of proximal splitting methods.  This latter situation is particularly relevant in statistical learning, where we have to minimize  an expected objective function from random samples. In this context, iterative algorithms, where only one gradient estimate is used in each step, are often referred to as online learning algorithms. More generally, the situation where only stochastic gradient estimates are available is important in stochastic optimization, where iterative algorithms can be seen as  a form of stochastic approximation.  
Finally, stochastic gradient approaches are considered in the incremental optimization of an objective function which is the sum of many terms, e.g. the empirical risk in machine learning   \cite{Bert11}, see Section \ref{ss:amm} for a detailed discussion.  In the next section we describe our contribution in the context of the state of the art.

%%%%%%%%%%%%%%%%%%%%%%%%%%%%%%%%%%%%%%%
\subsection{Contribution and Previous Work}
%%%%%%%%%%%%%%%%%%%%%%%%%%%%%%%%%%%%%%

The study of  stochastic approximation methods originates in the classical work of \cite{RobMon51}, 
and assumes the objective function to be smooth and strongly convex;  
the related literature is vast (see e.g. \cite{Ben90,Nem09,De11} and references therein). An improvement of the original 
stochastic approximation method, based on averaging of the trajectories and larger step-sizes, is proposed by \cite{Nem83} 
and \cite{PolJud92}. 
More recently, one  can recognize  two main approaches to solve
general nonsmooth convex  stochastic optimization. The first one uses different versions of mirror descent 
stochastic approximation, based on projected subgradient averaging techniques \cite{Jud11,Nem09,Lan09,ShaZha12}.  
Similar methods have been extensively studied also  in the machine learning community in the context of online learning, where the proof of convergence of the average of the iterates is often based  on regret analysis and, the so called, online-to-batch conversion \cite{Zin03,HazAgaKal07,hazan,RakShaSri12}.   
The second line of research is based on stochastic  variants of accelerated proximal   gradient descent \cite{Hu,Ghadimi12,Jud14,bottou2005line,Salev07,Shalev08,zhang2008multi}.

The algorithm we consider is also a stochastic extension of proximal gradient descent, but 
corresponds to its basic version with no acceleration. Indeed, as discussed below, a main question we consider  is  if  accelerated methods yield any advantage in the stochastic case.  
The FOBOS algorithm  in  \cite{Duchi09} is   the closest approach to the one we consider,  the main two differences being 
1) we consider an additional relaxation step which may lead to accelerations, and especially 2) 
we do not  consider averaging of the iterates. This latter point is important, since   averaging can have a detrimental effect. Indeed, non-smooth problems often arise in applications where sparsity of the solution is  of interest, and  it is easy to see that averaging prevent the solution to be sparse \cite{LinChePen14,Xiao}.  Moreover, as noted in \cite{RakShaSri12} and \cite{ShaZha12}, averaging can have a  negative impact on the convergence rate in the strongly convex case. Indeed,  in this paper we improve the error bound  in \cite{Duchi09} in this latter case.

Our study is developed in  an infinite dimensional setting, where we  focus on almost sure 
convergence of the iterates and non asymptotic bounds on their expectation.
Considering iterates convergence  is standard in optimization theory and  often considered in machine learning when sparsity based learning  is studied \cite{BulVan11}. The theoretical analysis in the paper is divided in two parts. In the first, we study convergence in  expectation in the strongly convex case, generalizing the results in \cite[Section 3]{bach} to  the nonsmooth case. We provide  a non-asymptotic analysis of stochastic proximal gradient descent where the bounds depend explicitly on the  parameters of the problem.  Interestingly,  we obtain, in the strongly convex case,  the same   $O(1/n)$ error bound   that can be obtained from the optimal rate of convergence for function values as achieved by accelerated methods, see e.g. \cite{Ghadimi12}.  This result  (confirmed by  numerical simulations) suggests that, unlike in the deterministic setting, in stochastic optimization acceleration does not have an impact on the rate of convergence.    In the second part,  we establish  almost sure convergence. Our results generalize  to the composite case the analysis of the stochastic  projected subgradient algorithm in a Hilbert space \cite{Barty07} (see also \cite{Bennar07,Monnez06}).  
Our analysis is  based on a novel extension of the analysis  of proximal methods with exact gradient,  based on considering  random quasi-Fej\'er sequences \cite{Ermol68}.  This approach allows to consider assumptions on the stochastic estimates of the gradients which are 
more general than those considered in previous work, and does not require  boundedness of the iterates. 

We  note that  a recent technical report \cite{AtcForMou14} also analyzes a stochastic proximal gradient method  (without the relaxation step) and its accelerated variant. Almost sure convergence of the iterates (without averaging) is proved under uniqueness of the minimizer, but under  assumptions different from ours:  continuity of  the objective function-- thus excluding constrained smooth optimization-- and boundedness the iterates. Convergence rates for the iterates without averaging are derived, but only  for the accelerated method.
Finally, we note that convergence of the iterates of stochastic proximal gradient has been recently  obtained from the  analysis of convergence of stochastic fixed point algorithms presented in the recent preprint \cite{ComPes14}. However, this latter results is derived from  summability assumptions on the errors  of  the stochastic estimates which are usually not satisfied in the machine learning setting.

The paper is organized as follows. In Section \ref{sec:psae} we introduce composite optimization and 
the stochastic proximal gradient algorithm, along with some relevant special cases.  
In Section \ref{s:alconv}, we study convergence in expectation and almost surely that we prove in  Section \ref{sec:proofs}.
Section \ref{s:exp} describes some numerical tests comparing the stochastic projected gradient algorithm with state of the
art stochastic first order methods. 
The proofs of auxiliary results are found in  Appendix \ref{s:bg}.

\paragraph{\bf Notation and basic definitions}
Throughout, $(\boldsymbol{E, \mathcal{A},P})$ is a probability space,
$\NN^* = \NN\backslash\{0\}$, and 
 $\HH$ is  a real separable Hilbert space.
 We use the notation
$\scal{\cdot}{\cdot}$ and $\|\cdot\|$ for the scalar product and the associated norm
in $\HH$. The symbols $\weakly$ and $\to$ denote, respectively, 
 weak and strong convergence. 
The class of lower semicontinuous convex functions  
$f\colon\HH\to\left]-\infty,+\infty\right]$ such 
that $\dom f=\menge{x\in\HH}{f(x) < +\infty}\neq\emp$, 
is denoted by $\Gamma_0(\HH)$. The proximity operator of $f\in\Gamma_0(\HH)$ is 
 \begin{equation}
\label{e:prox}
\prox_f\colon\HH\to\HH,\quad \prox_f(w)=\underset{v\in\HH}{\argmin}\: f(v) + \frac12\|w-v\|^2.
\end{equation}
Throughout this paper, we assume implicitly that the closed-form expressions of the 
proximity operators to be  available. 
We refer to \cite{livre1,Com11} for the closed-form expression of a wide class of 
functions, see \cite{MRSVV10} for examples in machine learning.
 Given  a random variable $X$, we denote by $\E[X]$ its expected value, and by
$\sigma(X)$ the  $\sigma$-field generated by $X$. The conditional expectation
of $X$ given a $\sigma$-algebra $\mathcal{A}\subset \boldsymbol{\mathcal{A}}$ is denoted by 
$\E[X|\mathcal{A}]$. The conditional expectation of $X$ given $Y$ is denoted by $\E[X|Y]$.
A filtration of  $\boldsymbol{\mathcal{A}}$ is an increasing sequence $(\mathcal{A}_n)_{n\in\NN^*}$
 of sub-$\sigma$-algebras of $\boldsymbol{\mathcal{A}}$.
A  $\HH$-valued random process is a sequence of random variables $(X_n)_{n\in\NN^*}$ taking values in $\HH$.
The shorthand notation `a.s.' stands for  `almost sure'.    
\section{Problem setting and examples}\label{sec:psae}
In this section, we  introduce the composite convex optimization problem, the stochastic proximal method we study,  
and discuss some  special cases of the framework we consider. 
\subsection{Problem}
Composite optimization problems are defined as 
the  problem of minimizing  the sum of a smooth convex function and a possibly nonsmooth convex function.
Here we assume that the latter is proximable, that is the proximity operator \eqref{e:prox} is available in closed form or can be easily computed.

\begin{problem}
\label{prob2}
Let $\nsm\in \Gamma_0(\HH)$, let $\beta \in ]0,\infty[$, and 
let $\smo\colon\HH\to\RR $ be convex and differentiable,  
with a $\beta$-Lipschitz continuous gradient.
The problem is to
\begin{equation}
\label{e:prob2}
  \underset{w\in\HH}{\text{minimize}}\; \obj(w)=\smo(w)+\nsm(w),
\end{equation}
under the assumption that the set of solutions to \eqref{e:prob2} is non-empty.
\end{problem} 
As  mentioned in the introduction, problems with this composite structure has 
been recently extensively studied in convex 
optimization. In particular, the class of splitting methods, which decouple the contribution of the
smooth term and the nonsmooth one, received a lot of attention  \cite{livre1}.
Within the class of splitting methods, in this paper we study the following stochastic 
proximal gradient (SPG) algorithm.

\begin{algo}[SPG] 
\label{a:main}
Let $(\gamma_n)_{n\in\NN^*}$ be a strictly positive sequence, 
let $(\lambda_n)_{n\in\NN^*}$ be a sequence in $\left[0,1\right]$, and
let $(\mathrm{G}_n)_{n\in\NN^{*}}$ be a $\HH$-valued random process such that
$(\forall n\in\NN^*)$ $\E[\|\mathrm{G}_n\|^2]<+\infty$. 
 Fix $w_1$ a $\HH$-valued integrable vector with $\E[\|w_1\|^2] < +\infty$ and set
\begin{equation}
\label{e:main1}
(\forall n\in \NN^{*})\quad
\begin{array}{l}
\left\lfloor
\begin{array}{l}
z_n = w_n- \gamma_n \mathrm{G}_n\\
y_{n} = \prox_{\gamma_n \nsm}z_n\\
w_{n+1} = (1-\lambda_n) w_{n} + \lambda_ny_{n}.\\
\end{array}
\right.\\[2mm]
\end{array}
\end{equation}
\end{algo}
Algorithm \ref{a:main} is  a stochastic version of the proximal 
forward-backward splitting  \cite{siam05}, where
we replace  the exact gradient by a stochastic element.  More specifically, 
if, for every $n\in\NN^*$, $\mathrm{G}_n = \nabla\smo(w_n)$, our algorithm  reduces to 
the one in \cite{siam05}.
A stochastic proximal forward-backward splitting (FOBOS)
was firstly proposed in \cite{Duchi09} for minimizing 
the sum of two functions where one of them is proximable, 
and the other is convex and subdifferentiable.
Algorithm \ref{a:main}   generalizes  the FOBOS  algorithm,  by including a relaxation step, while assuming   the first component in \eqref{e:prob2} to be smooth.
As it is the standard, to ensure convergence of the  proposed algorithm,
we need  additional conditions on 
the random process $(\mathrm{G}_n)_{n\in\NN^*}$ as well as on the sequence of step-sizes
$(\gamma_n)_{n\in\NN^*}$.

\begin{condition}
The following conditions will be considered for the
 filtration $(\mathcal{A}_n)_{n\in\NN^*}$ with 
$\mathcal{A}_n = \sigma(w_1,\ldots,w_n)$.
\begin{enumerate}
\item[(A1)] 
\label{e:csss1}
 For every $n\in\NN^*$, $\E\left[\mathrm{G}_n|\mathcal{A}_n \right] = \nabla \smo(w_n)$.
\item[(A2)] 
\label{e:csss2}
 For every $n\in\NN^*$, there exist $\sigma\in\left]0,+\infty\right[$ and $\alpha_n \in \left]0,+\infty\right[$ such that
\begin{equation}
\E\left[\|\mathrm{G}_n -\nabla \smo(w_n)\|^2|\mathcal{A}_n\right] 
\leq \sigma^2(1+ \alpha_n\|\nabla \smo(w_n)\|^2 )
\end{equation}
\item[(A3)]  There exists $\epsilon\in\left]0,+\infty\right[$ such that 
$(\forall n\in\NN^*)\;  0<\gamma_n\leq \dfrac{1-\epsilon}{\beta( 1 + 2\sigma^2\alpha_n)} $ 
\item[(A4)] For any solution $\overline{w}$ of the problem \eqref{e:prob2}, set $(\forall n\in\NN^*)\;\chi^{2}_n =
\lambda_n\gamma_{n}^2\big(1 +2\alpha_{n}\|\nabla \smo(\overline{w})\|^2\big)$. Assume that
\begin{equation}
\label{e:suma}
 \sum_{n\in\NN^{*}}\lambda_n\gamma_n = +\infty
\quad \text{and}\quad 
 \sum_{n\in\NN^{*}}\chi^{2}_n < +\infty.
\end{equation}
\end{enumerate}
\end{condition}
Condition (A1) means that,  at each iteration $n$, $\mathrm{G}_n$ is an unbiased estimate of the gradient of  the smooth term. Condition (A2) has been 
considered in \cite{Barty07}.
It is weaker than typical conditions used in the analysis of 
stochastic (sub)gradient algorithms, namely boundedness of the sequence
$(\E[\|\mathrm{G}_n\|^2|\mathcal{A}_n])_{n\in\NN^*}$ (see \cite{Nem09}) or even boundedness 
of $(\|\mathrm{G}_n\|^2)_{n\in\NN^*}$ (see \cite{Duchi09}). 
We note that this last requirement on the entire space is not compatible with the assumption of strong convexity, 
because  the gradient is necessarily not uniformly bounded,  therefore the use of the more general condition (A2) is needed in this case.

Conditions such as (A3) and (A4) are, respectively, widely used in the 
deterministic setting  and in stochastic optimization. 
Assumption (A3) is more restrictive that the one usually assumed in the deterministic setting, that is $(\forall n\in\NN^*)\; \gamma_n\leq (2-\epsilon)/\beta$.
We also note that when $(\lambda_n)_{n\in\NN^*}$  
is bounded away from zero, and $(\alpha_n)_{n\in\NN^*}$ is bounded, 
(A4) implies (A3) for $n$ large enough.  
 The condition  $\sum_{n\in\NN^{*}}\chi^{2}_n<+\infty$ in  
Assumption $(A4)$ is  satisfied if $\big(\lambda_n\gamma_{n}^2\big(1 +2\alpha_{n}\big)\big)_{n\in\NN^{*}}$ 
is summable. 
Moreover, if $R=0$, it reduces to $\sum_{n\in\mathbb{N}^*}\lambda_n\gamma_{n}^2<+\infty$, 
since in this case $\nabla \smo(\overline{w})=0$ for every solution $\bar{w}$.
Finally, in our case, the step-size is required to converge to zero,
while it is typically bounded away from zero in the study of deterministic 
proximal forward-backward splitting algorithm \cite{siam05}.

\subsection{Special cases}

Problem \ref{prob2} covers a wide class of deterministic  as well as 
stochastic convex optimization problems, especially from machine  learning and signal processing, see e.g. \cite{siam05,siam07,Nem09,MRSVV10,glopridu} 
and  references therein.
The simplest case is when $\nsm$ is identically equal to 0, so that  
Problem \ref{prob2} reduces to the classic problem of 
finding a  minimizer of a convex differentiable  function from unbiased estimates of its gradients.
In the case when  $\nsm$ is the indicator function of a nonempty, convex, closed  set $C$, i.e. 
 \begin{equation}
\nsm(w) = \iota_C(w) 
= \begin{cases}
  0 &\text{$w\in C$},\notag\\
+\infty &\text{$w\not\in C$},
  \end{cases}
\end{equation}
then problem \eqref{e:prob2} reduces to a constrained minimization problem of the form 
\begin{equation*}
  \underset{w\in C}{\text{minimize}}\;  \smo(w),
\end{equation*}     
which is well studied in the literature, as mentioned in the introduction.
Below, we discuss in more detail some special cases of interest.

\begin{example}({\bf Minimization of an  Expectation}).\label{Ex:ExpMin}
 Let $\xi$ be a random vector with probability distribution $P$ supported on $\EEE$
and $F\colon\HH\times\EEE\to \RR$. Stochastic
gradient descent methods are usually studied in the case where $\HH$ is an euclidean space and  
\begin{equation*}
\smo(w) = \E\left[ F(w,\xi)\right] = \int_{\EE} F(w,\xi)dP(\xi),
\end{equation*}  
under the assumption that  $(\forall \xi \in \EEE)$ $F(\cdot,\xi)$ is  
a convex differentiable function with Lipschitz continuous gradient \cite{Nem09}. 
Let  $(\xi_n)_{n\in\NN^*}$ be independent copies of the random vector $\xi$. 
Assume that there is an oracle that, for each $(w,\xi) \in \HH\times \EE$, 
returns a vector $G(w,\xi)$ such that $\nabla \smo(w) = \E\left[G(w,\xi)\right]$.
By setting $(\forall n\in\NN^*)\; \mathrm{G}_n = G(w_n,\xi_n)$ and $\mathcal{A}_n=\sigma(\xi_1,\ldots,\xi_{n})$, 
then (A2) holds. This latter assertion follows from standard properties of conditional expectation, see e.g. \cite[Example 5.1.5]{Duret04}. 
\end{example}

\begin{example}({\bf Minimization of a Sum of Functions})\label{Ex:SumOfFunctions}
 Let $\nsm\in \Gamma_0(\HH)$,  let $m$ be a strictly positive integer. For every 
$i\in\{1,\ldots,m\}$, let $\smo_i$ be convex and differentiable, such that 
$\sum_{i=1}^m\smo_i$ has a $\beta$-Lipschitz continuous gradient,
for some $\beta\in \left]0,+\infty\right[$. The problem is to 
\begin{equation*}
  \underset{w\in \HH}{\text{minimize}}\;  \frac{1}{m}\sum_{i=1}^m\smo_i(w)+\nsm(w). 
\end{equation*}
This problem is a special case of Problem \ref{e:prob2} with $R = \sum_{i=1}^m \smo_i$, and
is especially of  interest when $m$ is very large and we know the exact gradient of each component $L_i$. The stochastic
 estimate of the gradient of $\smo$ is then defined as 
\begin{equation}
(\forall n\in\NN)\quad \mathrm{G}_n = \nabla \smo_{\operatorname{i}(n)}(w_n),
\end{equation}
where $(\operatorname{i}(n))_{n\in\NN^*}$ is a random process of independent random 
variables uniformly distributed on $\{1,\ldots,m\}$, see \cite{Bert97,Bert11}.
Clearly $(A1)$ holds. Assumption $(A2)$  specializes in this case to
\begin{equation}
(\forall n\in \NN^*)\quad \frac{1}{m}\sum_{i=1}^m\Big(\|\nabla L_i(w_n)-\frac{1}{m}\sum_{i=1}^m \nabla L_i(w_n)\|^2\Big)\leq\sigma^2(1+\alpha_n\|\nabla L(w_n)\|^2)\,.
\end{equation} 
If the latter is satisfied, then  SPG algorithm can be applied with a suitable choice of the 
stepsize.
\end{example}

Finally, in the next section,  we discuss how the above setting specializes to the context of machine learning. 
\subsection{Application to   Machine Learning}\label{ss:amm}

Consider  two measurable spaces $\X$ and $\Y$ and assume there is 
a probability measure $\rho$ on $\X\times \Y$. The measure $\rho$
is fixed but known only through a training set $\zz = (x_i,y_i)_{1\leq i\leq m} \in (\mathcal{X}\times\mathcal{Y})^m$ of 
samples i.i.d with respect to $\rho$.  Consider a loss function $\ell:\Y\times\Y\to \left[0,+\infty\right[$ and a hypothesis space $\HH$ 
of functions from $\X$ to $\Y$, e.g. a  reproducing kernel Hilbert space. A key problem in this context is   (regularized) empirical risk minimization, 
\begin{equation}
\label{e:cyc}
\underset{w\in \HH}{\text{minimize}}\;\frac{1}{m}\sum_{i=1}^m\ell(y_i,w(x_i))+
\nsm(w),
\end{equation}
The above problem can  be seen as an approximation of the  problem,
\begin{equation}
\label{e:learn1}
  \underset{w\in\HH}{\text{minimize}}\;   
  \int_{\X\times\Y} \ell(y,w(x))\,d\rho + \nsm(w).
\end{equation}
The analysis, in this paper, can be adapted to the machine learning setting in two different ways. The first, following 
Example \ref{Ex:ExpMin}, is to apply  the SPG algorithm to directly solve the regularized {\em expected}  loss 
minimization problem \eqref{e:learn1}. 
The second, following  Example \ref{Ex:SumOfFunctions}, is to apply  the SPG algorithm to  solve the regularized {\em empirical}  risk minimization problem \eqref{e:cyc}.

In either one of the above two  problems, the first term is differentiable if the loss functions is differentiable with respect to  
its second argument, examples being  the squared  or the logistic loss. For these latter  loss functions, and more generally 
for loss functions which are twice differentiable in their second argument,  it easy to see that the Lipschitz continuity of the 
gradient is satisfied if the maximum eigenvalue of the Hessian is bounded. The second term $R$ can be seen as a 
regularizer/penalty encoding some prior information about the learning problem. Examples of convex, non-differentiable penalties 
include sparsity inducing penalties such as the $\ell_1$ norm, as well as more complex structured sparsity penalties \cite{MRSVV10}. 
Stronger convexity properties can be obtained considering an {\em elastic net penalty} \cite{Zou05,DDR09}, that is adding  a small 
strongly convex term to the sparsity inducing penalty. Clearly, the latter term would not be necessary if the risk in Problem  
\ref{e:learn1} (or the empirical risk in \eqref{e:cyc}) is strongly convex. However, this latter requirement depends on the probability  
measure $\rho$ and is typically not satisfied  when considering  high  (possibly  infinite) dimensional settings. 

%%%%%
\section{Main results and discussion}\label{s:alconv}
%%%%%
In this section, we state and discuss the main results of the paper. 
We derive convergence rates of 
the proximal gradient algorithm (with relaxation) for stochastic minimization.
The section is divided in two parts. 
In the first one, Section \ref{sec:exp}, we focus on convergence in expectation.  
In the second one, Section \ref{sec:as}, we study almost sure convergence of the 
sequence of iterates. In both cases, additional convexity conditions on the objective function are  required 
to derive convergence results. The proofs are deferred to Section \ref{sec:proofs}.  
%%%%%%
\subsection{Convergence in Expectation of SPG algorithm}\label{sec:exp}
%%%%%%
In this section, we  denote by $\overline{w}$ a solution of Problem \ref{e:prob2}  and provide an explicit 
non-asymptotic bound on  $\E[\|w_n-\overline{w}\|^2]$. This result generalizes to the nonsmooth case
the bound obtained  in \cite[Theorem 1]{bach} for stochastic gradient descent. 
The following assumption is considered throughout  this section.

\begin{assumption}\label{A:strcon}
The function $\smo$ is $\mu$-strongly convex and $\nsm$ is $\nu$-strongly convex,
for some $\mu\in\left[0,+\infty\right[$ and $\nu\in\left[0,+\infty\right[$, with $\mu+\nu>0$. 
\end{assumption}
Note that, we do not assume  both $\smo$ and $\nsm$ to be strongly convex,
indeed the constants $\mu$ and $\nu$ can be zero, but  require  that only one of the two is.
This implies that $\overline{w}$ is the unique solution of Problem \ref{e:prob2}.

In the statement of the following theorem, we will use the
family of functions $(\varphi_c)_{c\in\RR}$ defined by setting, for every $c\in \RR$,
\begin{equation}
\label{e:bach1}
 \varphi_{c}\colon \left]0,\pinf\right[\to \RR\colon
t\mapsto 
\begin{cases}
(t^{c}-1)/c& \text{if $c \not=0$};\\
 \log t& \text{if $c =0$}.
\end{cases}
\end{equation}
This family of functions arises in Lemma \ref{l:ocs} in the Appendix and are useful to bound the sum of the
stepsizes. 
\begin{theorem}
\label{t:1}
Assume that conditions $(A1), (A2)$, $(A3)$ and Assumption \ref{A:strcon}  are satisfied. 
Suppose that there exist $\underline{\lambda}\in\left]0,+\infty\right[$ and $\overline{\alpha}\in\left]0,+\infty\right[$ 
such that 
\begin{equation}
\inf_{n\in\NN^*} \lambda_n \geq  \underline{\lambda}\quad\text{and}\quad\sup_{n\in\mathbb{N}^*}\alpha_n \leq \bar{\alpha}\,.
\end{equation}
Let $c_1\in\,]0,+\infty[$ and let $\theta \in \left]0,1\right]$. Suppose that, for every $n\in\NN$,  
$ \gamma_n =  c_1n^{-\theta}$. Set  
\begin{equation}
t = 1-2^{\theta-1}, \qquad c=\frac{2c_1\underline{\lambda}(\nu+\mu\varepsilon)}{(1+\nu)^2}, \quad and
\quad \tau=\frac{2\sigma^2c_1^2 (1+\overline{\alpha} \|\nabla \smo(\overline{w})\| )}{c^2 }.
\end{equation}
Let $n_0$ be the smallest integer such that $n_0>1$, and $\max\{c,c_1\} {n_0}^{-\theta}\leq 1.$
Then, by setting   
$$(\forall n\in\NN^{*})\quad\quad s_n = \E\left[\|w_{n}-\overline{w}\|^2\right],$$ 
we have, for every $n\geq 2n_0$,
\begin{equation}
\label{eq:Est1}
 s_{n+1} \leq 
\begin{cases}
\Big(\tau c^2 \varphi_{1-2\theta}(n)
 + s_{n_0}\exp\Big(\dfrac{cn_{0}}{1-\theta}\Big) \Big)
\exp\Big(\dfrac{-ct(n+1)^{1-\theta}}{1-\theta} \Big)
+ \dfrac{ 2^{\theta}\tau c}{(n-2)^{\theta}}
 &\text{if $\theta\in\left]0,1\right[$,}\\
 s_{n_0}\Big(\dfrac{n_0}{n+1}\Big)^{c}+ 
\dfrac{2^c\tau c^2 }{(n+1)^{c}}\varphi_{c-1}(n) &\text{if $\theta = 1$.}
\end{cases}
\end{equation}
\end{theorem}

In Theorem \ref{t:1}, the dependence on the strong convexity constants is hidden in the 
constant $c$. Taking into account \eqref{eq:Est1}, we can write more explicitly  the asymptotic 
behavior of the sequence $(s_n)_{n\in\NN^*}$.

\begin{corollary}\label{cor:1} Under the same assumptions and with 
the same notation of Theorem \ref{t:1}, the following holds
\begin{equation}
\E[\|w_{n}-\overline{w}\|^2]=\begin{cases}O(n^{-\theta})& \text{if $\theta\in\,]0,1[$},\\ 
O(n^{-c})+O(n^{-1}) &\text{if $\theta=1$}.
\end{cases}
\end{equation}
Thus, if $\theta =1$ and $c_1$ is chosen such that $c>1$, then $\E[\|w_{n}-\overline{w}\|^2] = O(n^{-1})$. 
In particular, if $\theta=1$,  $\lambda_n= 1 = \underline{\lambda}$ for every $n\in\NN^*$, 
and $c_1= (1+\nu)^2/\underline{\lambda}(\nu+\mu\varepsilon) > 2$,  then
$c=2$, $n_0=\max\{2,c_1\}$, and  
\begin{equation}
\E[\|w_{n}-\overline{w}\|^2]\leq \dfrac{n_{0}^2\E[\|w_{n_0}-\overline{w}\|^2] }{(n+1)^2}+ 
\dfrac{8\sigma^2(1+\overline{\alpha} \|\nabla \smo(\overline{w})\|)(1+\nu)^4}{\underline{\lambda}^2\left(\mu \epsilon+\nu\right)^2}
\end{equation}
\end{corollary}

Theorem \ref{t:1} is the extension to the nonsmooth case of \cite[Theorem 1]{bach}, in particular,  when $\nsm=0$,
we obtain the same bounds. Note however that the assumptions
on the stochastic approximations of the gradient of the smooth part are different. In particular,
we replace the boundedness condition at the solution and the  Lipschitz continuity assumption on $(\mathrm{G}_n)_{n\in\NN^*}$
with assumption (A2). 
As can be seen from Corollary \ref{cor:1}, the fastest asymptotic rate corresponds to  $\theta=1$
and it is the same obtained in the smooth case in \cite[Theorem 2]{bach}. Note that this rate depends on the asymptotic behavior of the step-size, but also on the constant $c$, which in turns depends on $c_1$. 
As   pointed out in \cite{Nem09}, see also in \cite{bach}, this choice is critical, because too small choices
of $c_1$ affect the convergence rates, and too big choices  influence significantly  the value of the constants
 in the first term of \eqref{eq:Est1}. In particular, as can be readily seen in Corollary \ref{cor:1}, 
 the choice is determined by the strong convexity constants. Moreover, the dependence on the strong 
 convexity constant shown in Corollary \ref{cor:1} is of the same type of the one obtained in the regret minimization
 framework by \cite{hazan}.

There are other stochastic first order methods achieving the same rate of convergence for the iterates 
in the strongly convex case, see e.g. \cite{AtcForMou14,hazan,Ghadimi12, Jud14,Xiao,LinChePen14}. Indeed, the rate we obtain is the rate 
that can be  obtained by the optimal (in the sense of \cite{Nem83}) convergence rate on the function values.
Among the mentioned methods those in \cite{AtcForMou14,Ghadimi12,LinChePen14} belong to the
class of accelerated proximal gradient methods. Our result shows that, in the strongly convex case, the rate of
convergence of the iterates is the same in the accelerated and non accelerated case. 
In addition, if sparsity is the main interest, we highlight that many of the algorithms discussed above 
(e.g. \cite{AtcForMou14,Ghadimi12,hazan,Xiao})
involve some form of averaging or linear combination which prevent sparsity of the iterates, as it is discussed in
\cite{LinChePen14}. Our result shows that in this case averaging is not needed, since the iterates themselves are convergent. 

We next compare in some  detail our results with those obtained for the FOBOS algorithm  in \cite{Duchi09} and to the
stochastic proximal gradient in \cite{AtcForMou14}. 
There are a few difference in the settings considered. In particular, 
convergence of the average of the iterates  with respect to the 
function values is considered in \cite{Duchi09} assuming 
uniform boundedness of the iterations and the subdifferentials. 
The space  $\HH$ is assumed to be  finite dimensional, 
though the analysis might be extended to infinite dimensional spaces.
Finally, the optimal stepsize in \cite{Duchi09}  depends explicitly on the radius of the ball containing  
the iterates, which in general might not be available. 
Our convergence results consider convergence of the iterates (with no averaging) 
and hold in an infinite dimensional setting, without boundedness assumptions. 
The non asymptotic rate $O(n^{-1})$ which we obtain  for the iterates improves  
the $O((\log n)/n)$ rate derived from \cite[Corollary 10]{Duchi09} for  the average 
of the iterates. However, it should be noted that  convergence of the objective 
values is studied  in \cite{Duchi09} also for  the non strongly convex case. 
SPG  (without relaxation) has been recently studied in \cite{AtcForMou14}. 
Also in this case the authors assume a priori boundedness of the iterates and prove
convergence of the averaged sequence.

Theorem \ref{t:1} is also comparable with deterministic stochastic proximal 
forward-backward algorithm with errors \cite{siam05}. 
On the one hand, we allow the errors to satisfy assumption (A2), 
while in the  deterministic case the errors in the computation
of the gradient should decrease  to zero sufficiently fast. On the other hand, 
we require asymptotically vanishing (and smaller, according to (A3)) step-sizes,
while, in the deterministic case, the step-size is bounded from below. 
Finally, if $\obj$ is  continuous,  in the setting of Theorem \ref{t:1}, 
it holds $\obj(w_n)\to \min_\HH T$. Moreover, if $\obj$ is Lipschitz continuous, 
then $\obj(w_n)-\min_\HH T=O(n^{-1/2})$ and if $\obj$ is differentiable with Lipschitz 
continuous gradient, $\obj(w_n)-\min_\HH T=O(n^{-1})$.
 
%%%%%%
\subsection{Almost sure convergence of SPG algorithm}\label{sec:as}
%%%%%%
In this section, we focus on almost sure convergence of  {SPG} algorithm.
This kind of convergence of the iterates is the one traditionally studied in the stochastic optimization
literature.  Depending on the convexity properties of the function
$\smo$, we get two different convergence properties. 
The first theorem requires uniform convexity of $\smo$ at the solution.
\begin{theorem} 
\label{t:2} 
 Suppose that the conditions $(A1), (A2)$, $(A3)$, and $(A4)$ are satisfied.
Let $(w_n)_{n\in\NN^{*}}$ be a sequence generated by Algorithm \ref{a:main} and
assume that  $\smo$ is uniformly convex at $\overline{w}$.  
Then  $w_n \to \overline{w}$ a.s.
\end{theorem}
If we relax the strong convexity assumption, we can still prove weak convergence of 
a subsequence in the strictly convex case, provided an additional
regularity assumption holds. 
\begin{theorem} 
\label{t:3} 
Suppose that
the conditions $(A1), (A2)$, $(A3)$, and $(A4)$ are satisfied.
Let $(w_n)_{n\in\NN^{*}}$ be a sequence generated by Algorithm \ref{a:main}.
Assume that $\smo$ is strictly convex, and let $\overline{w}$ be the unique solution of
Problem \ref{e:prob2}. If $\nabla \smo$  is  weakly continuous, 
then there exists a subsequence $(w_{t_n})_{n\in\NN^{*}}$ such that
$w_{t_n}\weakly\overline{w}$ a.s.
\end{theorem}

With respect to the previous section, 
here we make the additional assumption  $(A4)$
on the summability of the sequence of 
step-sizes multiplied by the relaxation parameters. 
For stochastic gradient algorithm without relaxation, i.e, $\nsm=0$ and, for every $n\in\NN^*$
$\lambda_n=1$, assumption (A4) coincides 
with the classical step-size condition $\sum_{n\in\NN^*} \gamma_n=+\infty$ 
and  $\sum_{n\in\NN^*} \gamma_n^2<+\infty$ which guarantees a sufficient but not too
fast decrease of the step-size (see e.g. \cite{Bert00}).
Assumption $(A2)$ has been considered in the context of stochastic gradient descent in \cite{Bert00}. 
Note that under such a condition, the variance of the stochastic approximation is allowed to grow with
$\|\nabla \smo(w_n)\|$.

As mentioned in the introduction, the study of  almost sure convergence is classical.
An analysis of a stochastic projected subgradient algorithm in an infinite dimensional Hilbert space  can be found in \cite{Barty07}. Theorem \ref{t:2} can be seen as an extension of \cite[Theorem 3.1]{Barty07},
where  the case where $\nsm$ is an indicator function is considered. 
Our approach is based on random quasi-Fej\'er sequences, and on probabilistic quasi martingale  techniques \cite{Met82}.

\begin{remark} 
\label{r:3}
If $\smo$ is assumed to be only strictly convex and its gradient is not weakly continuous, 
Theorem \ref{t:3} does
not ensure weak convergence of any  subsequence of $(w_n)_{n\in\NN^*}$. 
However, if  the sequence of function values
 $(T(w_n))_{n\in\NN^{*}}$ 
converges to the minimum of $T$, then $w_n\weakly \overline{w}$ a.s. 
This happens (see \cite{Barty07}) when $\nsm = \iota_V$ for some closed subspace $V$ of $\HH$, or
when  $\nsm = \iota_C$ for some non-empty closed convex $C$ of $\HH$, and
there exists a bounded function $h\colon \RR\to \RR$ such that 
$ (\forall n\in\NN)\quad \E[\|\mathrm{G}_n-\nabla \smo(w_n)\||\mathcal{A}_n] \leq h(\|\nabla \smo(w_n)\|).$
\end{remark}
The proof of Remark \ref{r:3} can be found in the next section.

%%%%%
\section{Proof of the Main Results}\label{sec:proofs}
%%%%%
We start by recalling the firmly non-expansiveness of the proximity operator and the 
Baillon-Haddad Theorem (see \cite[Theorem 18.15]{livre1}).

\begin{lemma}\cite[Lemma 2.4]{siam05}
\label{l:firm} 
Let $\nsm\in\Gamma_0(\HH)$. Then the proximity of $\nsm$ is firmly non-expansive, i.e., 
\begin{equation}
\label{e:firm2}
(\forall w\in\HH)(\forall u\in\HH)\quad \|\prox_\nsm w-\prox_\nsm u\|^2 
\leq \|u-w\|^2 - \|(w-\prox_\nsm w)-(u-\prox_\nsm u)\|^2.
\end{equation}
\end{lemma}

\begin{definition}\cite[Definition 4.4]{livre1}
\label{d:coco}
Let $B\colon\HH\to\HH$, and let $\alpha\in\RPP$. Then $B$ is $\alpha$-cocoercive if 
\begin{equation}
(\forall u\in\HH)(\forall w\in\HH)\; \langle u-w,Bu-Bw \rangle \geq \alpha \| Bu-Bw\|^2
\end{equation}
\end{definition}

\begin{lemma}[\bf Baillon-Haddad theorem]
Let $\smo \in \Gamma_0(\HH)$ be a convex differentiable function with $\beta$ Lipschitz gradient. 
Then, $\nabla \smo$ is $\beta^{-1}$-cocoercive.
 \end{lemma}

We next state the following lemma; see also \cite[Lemma 5, Chapter 2.2]{Pol87} and \cite{bach}. We will use the family of 
functions $(\varphi_{c})_{c\in\RR}$ defined in \eqref{e:bach1}. For completeness, the proof is given in the Appendix.
\begin{lemma}
\label{l:ocs}
 Let $\alpha\in\left]0,1\right]$, 
and let $c$ and $\tau$ be in $]0,+\infty[$, 
let  $(\eta_n)_{n\in\NN^{*}}$ be a strictly positive 
sequence defined by $(\forall n\in \NN^*)\; \eta_n = c n^{-\alpha}$. 
Let $(s_{n})_{n\in\NN^*}$ be such that 
\begin{equation}
\label{e:iter}
(\forall n\in \NN^*)\quad 0 \leq s_{n+1} \leq (1-\eta_n) s_n + \tau\eta_{n}^2.
\end{equation}
Let $n_0$ be the smallest integer such that $\eta_{n_0}\leq 1$ and 
set $t=  1-2^{\alpha-1} \geq  0$.
Then, 
for every $n\geq 2n_0$, 
\begin{equation}
 s_{n+1} \leq 
\begin{cases}
\Big(\tau c^2 \varphi_{1-2\alpha}(n)
 + s_{n_0}\exp\Big(\frac{cn_{0}^{1-\alpha}}{1-\alpha}\Big) \Big)
\exp\Big(\frac{-ct(n+1)^{1-\alpha}}{1-\alpha} \Big)
+ \frac{\tau 2^{\alpha}c}{(n-2)^{\alpha}}
 &\text{if $\alpha\in\left]0,1\right[$,}\\
 s_{n_0}\big(\frac{n_0}{n+1}\big)^{c}+ 
\frac{\tau c^2}{(n+1)^{c}}(1+ \frac{1}{n_0})^{c}\varphi_{c-1}(n) &\text{if $\alpha = 1$.}
\end{cases}
\end{equation}
\end{lemma}

We start with a technical result, giving some bounds that will be repeatedly used.

\begin{proposition}
\label{p:bb}
Consider the setting of the SPG algorithm and let $\overline{w}$ be a solution of Problem \ref{e:prob2}.
Suppose that conditions (A1), (A2), and (A3) are satisfied. 
Then the following hold:
\begin{enumerate} 
\item 
\label{e:est1t}
$(\forall n\in\NN^{*}) \quad \|w_{n+1}-\overline{w}\|^2 \leq (1-\lambda_n)\|w_n-\overline{w}\|^2
+ \lambda_n \|y_n-\overline{w}\|^2.$
\item 
\label{p:bbii}
Set 
\begin{equation}
\label{eq:un}
(\forall n\in\NN^*)\quad  u_n = w_n-y_n -\gamma_n(\mathrm{G}_n-\nabla\smo(\overline{w})).
\end{equation}
Then, for every  $n\in\NN^{*}$
\begin{equation}
\label{e:est2}
\|y_n-\overline{w}\|^2 \leq \|w_n-\overline{w}\|^2 
-2\gamma_n\scal{w_n-\overline{w}}{\mathrm{G}_n - \nabla\smo(\overline{w})}+ \gamma^{2}_n\|\mathrm{G}_n - \nabla\smo(\overline{w}) \|^2-\|u_n\|^2.
\end{equation}
\item
\label{p:bbiii}
For every  $n\in\NN^{*}$
\begin{align}
\label{e:est5}
 \nonumber\E[\|y_n -\overline{w}\|^2] \leq &
\left(\E\left[\|w_n-\overline{w}\|^2\right]-2\gamma_n\left(1-\gamma_n\beta(1+2\sigma^2\alpha_n)\right)\cdot \right.\\
& \cdot\E[\scal{w_n-\overline{w}}{\nabla\smo(w_n)-\nabla \smo(\overline{w})}]
+2\gamma_n^2\sigma^2 (1+\alpha_n\|\nabla\smo(\overline{w})\|^2)\bigg).
\end{align}
\end{enumerate}
\end{proposition}
\begin{proof}
\ref{e:est1t}: Follows from convexity of $\|\cdot\|^2$. 

\ref{p:bbii}: We have
\begin{equation}
(\forall n\in\NN^{*})\quad \overline{w} = 
\prox_{\gamma_n\nsm}(\overline{w}-\gamma_n\nabla\smo(\overline{w})).
\end{equation}
Moreover, since $\prox_{\gamma_n\nsm}$ is firmly non-expansive by Lemma \ref{l:firm}
\begin{equation}
\|y_n-\overline{w}\|^2
 \leq \|(w_n-\overline{w}) - \gamma_n (\mathrm{G}_n - \nabla\smo(\overline{w})) \|^2-\|u_n\|^2\notag\\
\end{equation}
and the statement follows.

\ref{p:bbiii}: Note that, for every $n\in\NN^*$, we have that $w_n$ and $\mathrm{G}_n$ 
are measurable with respect to $\mathcal{A}$ since they
are $\mathcal{A}_n$ measurable and by definition $\mathcal{A}_n\subset\boldsymbol{\mathcal{ A}}$. 
The same holds for $z_n$, for it is 
the difference of two measurable functions. We next show by induction that $(\forall n\in\NN^*)$  
$\|w_n\|^2$ is integrable. First, $\|w_1\|^2$ is integrable  by assumption. 
Then, assume by inductive hypothesis that
$\|w_n\|^2$ is integrable. 
Then so is $\|z_n\|^2$, for $\mathrm{G}_n$ is square integrable by assumption. Moreover, 
 $\|w_{n+1}\|^2\leq \lambda_n\|y_n\|^2
+2(1-\lambda_n) \|w_n\|^2\leq \|z_n\|^2+2\|\prox_{\gamma_n R}0\|^2 + \|w_n\|^2$, because $\prox_{\gamma_n \nsm}$ is nonexpansive. 
Therefore $\|w_{n+1}\|^2$ is integrable and hence so is $\|w_{n+1}\|$.
 This implies that $\E[\scal{w_n-\overline{w}}{\mathrm{G}_n - \nabla\smo(\overline{w})}]<+\infty$ and 
$\E[\|\mathrm{G}_n-\nabla\smo(\overline{w})\|]<+\infty$. 
Therefore,
using  assumption (A1), we obtain  
\begin{alignat}{2}
\label{e:est3}
 (\forall n\in\NN^{*})\quad \E[\scal{w_n-\overline{w}}{\mathrm{G}_n - \nabla\smo(\overline{w})}] 
&= \E[\E[\scal{w_n-\overline{w}}{\mathrm{G}_n - \nabla\smo(\overline{w})}|\mathcal{A}_n ]\notag\\
&= \E[\scal{w_n-\overline{w}}{\E[ \mathrm{G}_n - \nabla\smo(\overline{w})|\mathcal{A}_n]}]\notag\\
&=\E[\scal{w_n-\overline{w}}{\nabla\smo(w_n) - \nabla\smo(\overline{w})}].
\end{alignat}
Moreover, using the assumption (A2), we  have 
 \begin{alignat}{2}
\label{e:est4}
  \E[\|\mathrm{G}_n - &\nabla\smo(\overline{w}) \|^2] \leq 2 \E[\|\nabla\smo(w_n) - \nabla\smo(\overline{w}) \|^2]
+ 2 \E[\|\mathrm{G}_n - \nabla\smo(w_n)\|^2]\notag\\
&\leq  2 \E[\|\nabla\smo(w_n) - \nabla\smo(\overline{w}) \|^2] + 2\sigma^2(1+ \alpha_n\E[\|\nabla\smo(w_n)\|^2])\notag\\
&\leq (2+ 4\sigma^2\alpha_n)\E[\|\nabla\smo(w_n) - \nabla\smo(\overline{w}) \|^2]
+ 2\sigma^2(1 + 2\alpha_n\|\nabla\smo(\overline{w})\|^2)\notag\\
&\leq  
{ (2 + 4\sigma^2\alpha_n)}{\beta}\E[\scal{w_n-\overline{w}}{\nabla\smo(w_n) - \nabla\smo(\overline{w})}] 
+ 2\sigma^2(1+ 2\alpha_n)\|\nabla\smo(\overline{w})\|^2,
 \end{alignat}
where the last inequality follows from the fact that $\nabla\smo$  is cocoercive since 
it is Lipschitz-continuous (by the Baillon-Haddad Theorem).
The statement then follows from \eqref{e:est2}, \eqref{e:est3}, and \eqref{e:est4}.
\end{proof}

We are now ready to prove Theorem \ref{t:1}.

\begin{proof}
Since  $\mu+\nu>0$, then $\smo+\nsm$ is strongly convex. 
Hence, Problem \eqref{e:prob2} has a unique minimizer $\overline{w}$. 
Since $\gamma_n \nsm$ is $\gamma_n\nu$-strongly convex, 
 by \cite[Proposition 23.11]{livre1} $\prox_{\gamma_n \nsm}$ is $(1+\gamma_n\nu)$-cocoercive, 
and then
\begin{align}
\nonumber (\forall n\in\NN^{*})\quad\|y_n - \overline{w}\|^2  \leq \frac{1}{(1+\gamma_n\nu)^2}
\|(w_n-\overline{w}) - \gamma_n (\mathrm{G}_n - \nabla \smo(\overline{w})) \|^2.
\end{align}
Next, proceeding as in the proof of Proposition \ref{p:bb}, we get an inequality analogue to \eqref{e:est5},
that is
\begin{align}
 \nonumber\E[\|y_n -\overline{w}\|^2] \leq &\frac{1}{(1+\gamma_n\nu)^2}
\left(\E\left[\|w_n-\overline{w}\|^2\right]-2\gamma_n\left(1-\gamma_n\beta(1+2\sigma^2\alpha_n)\right)\cdot \right.\\
 \label{eq:ymenw}& \cdot\E[\scal{w_n-\overline{w}}{\nabla\smo(w_n)-\nabla \smo(\overline{w})}]
+2\gamma_n^2\sigma^2 (1+\alpha_n\|\nabla\smo(\overline{w})\|^2)\bigg).
\end{align}
Since $\smo$ is strongly convex of parameter $\mu$, it holds 
$\scal{\nabla\smo(w_n)-\nabla\smo(\overline{w})}{w_n-\overline{w}}\geq \mu\|w_n-\overline{w}\|^2$.
Therefore, from \eqref{eq:ymenw}, using the $\mu$-strong convexity of $\smo$ and 
(A3), we get
 \begin{align}
\label{eq:strmon}
\E[\|y_n -\overline{w}\|^2] & \leq 
 \frac{1}{(1+\gamma_n\nu)^2}\bigg((1-2\gamma_n\mu\epsilon)\E\left[\|w_n-\overline{w}\|^2\right]
+ 2\sigma^2\chi_n^2\bigg).
\end{align}
Hence, by definition of $w_{n+1}$,
\begin{align}
\label{e:consq}
 \E[\|w_{n+1}-\overline{w} \|^2]  
&\leq \bigg(1-\frac{\lambda_n\gamma_n(2\nu + 
\gamma_{n}\nu^2+2\mu\epsilon)}{(1+\gamma_n\nu)^2}\bigg)\E[\|w_n-\overline{w}\|^2]+  
\frac{ 2\sigma^2\chi^{2}_n}{(1+\gamma_n\nu)^2}.
\end{align}
Let $\gamma_n=c_1n^{-\theta}$ and fix $n\geq n_0$. Since $\gamma_n\leq \gamma_{n_0}=c_1 n_0^{-\theta}\leq 1$, we have
\begin{equation}\label{eq:l1}
\frac{\lambda_n\gamma_n(2\nu + \gamma_{n}\nu^2+2\mu\epsilon)}{(1+\gamma_n\nu)^2} \geq 
\frac{2\underline{\lambda}(\nu+\mu\varepsilon)}{(1+\nu)^2} \gamma_n=c n^{-\theta},
\end{equation}
where we set $c=c_1{2\underline{\lambda}(\nu+\mu\varepsilon)}/{(1+\nu)^2}$.
On the other hand, 
\begin{equation}\label{eq:l2}
\frac{ 2\sigma^2\chi^{2}_n}{(1+\gamma_n\nu)^2}\leq {2\sigma^2 (1+\overline{\alpha} \|\nabla \smo(\overline{w})\|)} c_1^2n^{-2\theta}\,. 
\end{equation}
Then, putting together \eqref{e:consq}, \eqref{eq:l1}, and \eqref{eq:l2}, we get
$ \E[\|w_{n+1}-\overline{w} \|^2]  
\leq (1-\eta_n)\E[\|w_n-\overline{w}\|^2]+ \tau\eta_n^2$,
with $\tau=2\sigma^2c_1^2 (1+\overline{\alpha} \|\nabla \smo(\overline{w})\| )/c^2$ and $\eta_n=c n^{-\theta}$.
Finally, \eqref{eq:Est1} follows from Lemma \ref{l:ocs}.
\end{proof}

In order to prove Theorem \ref{t:2}, we start by giving the definition of deterministic and random quasi-Fej\'er sequences.  
We denote   by $\ell_+^1(\NN)$ the set of summable sequences in  $\RP$. 

\begin{definition}\cite{Ermol68}
 Let $S$ be a non-empty subset of $\HH$ and let $(\varepsilon_n)_{n\in\NN^*}$ be a sequence in $\ell_{+}^1(\NN^*)$. 
Then,
\begin{itemize}
 \item[(i)]A sequence $(w_n)_{n\in\NN^*}$ in $\HH$ is deterministic quasi-Fej\'er monotone with respect to the target set $S$ if
\begin{equation}
 (\forall w\in S)(\forall n\in\NN^*)\quad \|w_{n+1}-w\|^2 \leq \|w_n-w\|^2 + \varepsilon_n.
\end{equation}
  \item[(ii)] 
 A sequence of random vectors $(w_n)_{n\in\NN^*}$ in $\HH$ is stochastic
 quasi-Fej\'er monotone with respect to the target set $S$ if $\E[\|w_1\|^2] < +\infty$ and
\begin{equation}
 (\forall w\in S)(\forall n\in\NN^*)\quad \E[\|w_{n+1}-w\|^2|\sigma(w_1,\ldots, w_n)] \leq \|w_n-w\|^2 + \varepsilon_n.
\end{equation}
\end{itemize}
\end{definition}
The following result has been stated in \cite{Barty07} without a proof. For the sake of completeness, 
 a proof is given in the Appendix.
\begin{proposition}{\rm \cite[Lemma 2.3]{Barty07}}
\label{p:fejer}
 Let $S$ be a non-empty closed subset of $\HH$, 
let $(\varepsilon_n)_{n\in\NN^{*}}\in\ell_{+}^{1}(\NN^{*})$.
Let  $(w_n)_{n\in\NN^{*}}$ be a sequence of random vectors in $\HH$ such that $\E[\|w_1\|^2] < +\infty$, and let
$\mathcal{A}_n=\sigma(w_1,\ldots,w_n)$. Assume that  
\begin{equation}
\label{e:fejer}
(\forall n\in\NN^{*})\quad \E[\|w_{n+1}-w\|^2|\mathcal{A}_n] 
\leq \|w_n-w\|^2 + \varepsilon_n 
\quad \text{a.s}.
\end{equation}
Then the following hold.
\begin{enumerate}
\item \label{p:fejeri}
Let $w\in S$. 
 Then,
$(\E[\|w_n -w\|^2])_{n\in\NN^{*}}$ 
converges to some $\zeta_w\in\RR$ and $(\|w_n-w\|^2)_{n\in\NN^{*}}$
converges a.s. to an integrable random vector $\xi_w$.
\item \label{p:fejerii}
 $(w_n)_{n\in\NN^{*}}$ is bounded a.s.
\item \label{p:fejeriii}
The set of weak cluster points of $(w_n)_{n\in\NN^{*}}$ is non-empty a.s.
\end{enumerate}
\end{proposition}

We next collect some convergence results that will be
useful in the proof of the main Theorem \ref{t:2}.  

\begin{proposition}
\label{p:1}
Suppose that (A1), (A2), (A3), and (A4) are satisfied.
Let $(w_n)_{n\in\NN^{*}}$ be a sequence generated by Algorithm \ref{a:main}.
Then, for any solution $\overline{w}$ of the problem \eqref{e:prob2}, the following hold:
\begin{enumerate}
\item \label{p:1i}
The sequence $(\E[\| w_n-\overline{w}\|^{2}])_{n\in\NN^{*}}$ converges to a finite value.
\item 
\label{t:2i}
The sequence $(\|w_n-\overline{w}\|^2)_{n\in\NN^{*}}$ 
converges a.s to some integrable random variable $\zeta_{\overline{w}}$.
\item \label{p:1iii}
$\sum_{n\in\NN^{*}}\lambda_n\gamma_n \E[\scal{w_n-\overline{w}}{\nabla\smo(w_n) - 
\nabla\smo(\overline{w})}] < \pinf$.
Consequently, 
\begin{equation*}
\varliminf_{n\to\infty}\E[\scal{w_n-\overline{w}}{\nabla\smo(w_n)-\smo(\overline{w})}] = 0
\quad \text{and} \quad
 \varliminf_{n\to\infty}\E[\|\nabla\smo(w_n)-\nabla\smo(\overline{w})\|^2] = 0.  
\end{equation*}
\item\label{p:1ii}  
$\sum_{n\in\NN^{*}}
 \lambda_n\E[\|w_n-y_n-\gamma_n(\mathrm{G}_n-\nabla\smo(\overline{w}))\|^2] < \pinf $
and 
$\sum_{n\in\NN^{*}} \lambda_n\E[\|w_n-y_n \|^2] < +\infty$.
\end{enumerate}
\end{proposition}

\begin{proof}
By Proposition~\ref{p:bb}\ref{e:est1t}-\ref{p:bbiii}, 
and by condition (A3), we get
\begin{alignat}{2}
\label{e:cons}
 \E[\|w_{n+1}&-\overline{w} \|^2] 
\leq (1-\lambda_n) \E[\|w_n-\overline{w}\|^2] + \lambda_n \E[\|y_n-\overline{w}\|^2] \notag\\
&\leq \E[\|w_n-\overline{w}\|^2] 
- 2\varepsilon \gamma_n\lambda_n \E[\scal{w_n-\overline{w}}{\nabla\smo(w_n)-\nabla\smo(\overline{w})}] 
 + 2\sigma^2\chi^{2}_n -\lambda_n \E[\|u_n\|^2]\notag\\
&\leq \E[\|w_n-\overline{w}\|^2]  + 2\sigma^2\chi^{2}_n,
\end{alignat}
where the last inequality follows by the monotonicity of $\nabla \smo$.

\ref{p:1i}:
Since the sequence $(\chi^{2}_n)_{n\in\NN^{*}}$ is summable by assumption (A4), we derive 
from \eqref{e:cons} that  $(\E[\|w_{n+1}-\overline{w} \|^2])_{n\in\NN^{*}}$ converges to a finite value.

\ref{t:2i}: 
We estimate the conditional expectation with respect to $\mathcal{A}_n$ 
 of each term in the right hand side of \eqref{e:est2}. 
Since $w_n$ is $\mathcal{A}_n$-measurable, we have 
\begin{equation}\label{eq:alm4}
 \E[\|w_n-\overline{w}\|^2| \mathcal{A}_n] = \|w_n-\overline{w}\|^2.
\end{equation}
Using assumption (A1),
\begin{alignat}{2}
\label{e:alm2}
 (\forall n\in\NN^{*})\quad \E[\scal{w_n-\overline{w}}{\mathrm{G}_n 
- \nabla\smo(\overline{w})}| \mathcal{A}_n] 
&= \scal{w_n-\overline{w}}{\E[ \mathrm{G}_n - \nabla\smo(\overline{w})| \mathcal{A}_n }\notag\\
&=\scal{w_n-\overline{w}}{\nabla\smo(w_n) - \nabla\smo(\overline{w})}. 
\end{alignat}
Next, note that $\nabla\smo(w_n)$ is $\mathcal{A}_n$-measurable by (A1), and therefore by $(A2)$, 
we get
 \begin{alignat}{2}
\label{e:alm3}
  \E[\|\mathrm{G}_n - &\nabla\smo(\overline{w}) \|^2| \mathcal{A}_n] 
\leq 2 \E[\|\nabla\smo(w_n) - \nabla\smo(\overline{w}) \|^2| \mathcal{A}_n]
+ 2 \E[\|\mathrm{G}_n - \nabla\smo(w_n)\|^2 |\mathcal{A}_n  ]\notag\\
&\leq  2\|\nabla\smo(w_n) - \nabla\smo(\overline{w}) \|^2
 + 2\sigma^2(1+\alpha_n\|\nabla\smo(w_n)\|^2)\notag\\
&\leq  2\|\nabla\smo(w_n) - \nabla\smo(\overline{w}) \|^2+
2\sigma^2(1+ 2\alpha_n  \|\nabla\smo(w_n) - \nabla\smo(\overline{w}) \|^2+ 2\alpha_n\|\nabla\smo(\overline{w})\|^2 )\notag\\
&\leq {(2 + 4\sigma^2\alpha_n )}{\beta} \scal{w_n-\overline{w}}{\nabla\smo(w_n) - \nabla\smo(\overline{w})} 
+ 2\sigma^2(1+ 2\alpha_n\|\nabla\smo(\overline{w})\|^2),
 \end{alignat}
where the last inequality follows from the cocoercivity
 of $\nabla\smo$. 
Taking the conditional expectation with respect to $\mathcal{A}_n$, and invoking \eqref{e:est2}, \eqref{eq:alm4}, 
\eqref{e:alm2}, and  \eqref{e:alm3}, we obtain, 
\begin{alignat}{2}
\label{e:concl2}
 \E[\|w_{n+1}-&\overline{w} \|^2|\mathcal{A}_n] 
\leq (1-\lambda_n) \|w_n-\overline{w}\|^2 
+ \lambda_n \E[\|y_n-\overline{w}\|^2|\mathcal{A}_n ] \notag\\
&\leq \|w_n-\overline{w}\|^2
- 2\gamma_n\lambda_n(1-{\beta\gamma_n(1+2\sigma^2\alpha_n)})\scal{\nabla\smo(w_n)-\nabla\smo(\overline{w})}{ w_n-\overline{w}}
\notag\\
&\quad + 2\sigma^2\chi^{2}_n-\lambda_n\E[ \|u_n\|^2|\mathcal{A}_n]\notag\\
&\leq \|w_n-\overline{w}\|^2
- 2\varepsilon \gamma_n\lambda_n \scal{\nabla\smo(w_n)-\nabla\smo(\overline{w})}{ w_n-\overline{w}} + 
2\sigma^2\chi^{2}_n-\lambda_n\E[ \|u_n\|^2|\mathcal{A}_n]\notag\\
&\leq \|w_n-\overline{w}\|^2 + 2\sigma^2\chi^{2}_n\,.
\end{alignat}
Hence, $(w_n)_{n\in\NN^{*}}$ is a random quasi-Fej\'er sequence with respect to the nonempty closed and convex set
$\Argmin \obj$. 

Taking into account that $\E[\|w_1\|^2]<+\infty$  by assumption, it follows from Proposition \ref{p:fejer}\ref{p:fejeri} that
$(\|w_n-\overline{w}\|^2)_{n\in\NN^{*}}$ converges a.s
to some integrable random variable $\zeta_{\overline{w}}$. 

\ref{p:1iii}: We derive from \eqref{e:cons} that 
\begin{equation}
\label{e:ese1}
\sum_{n\in\NN^{*}}
\gamma_n\lambda_n \E[\scal{w_n-\overline{w}}{\nabla\smo(w_n)-\nabla\smo(\overline{w})}] < +\infty.
\end{equation}
Since $\sum_{n\in\NN^{*}}\lambda_n\gamma_n = +\infty$, we obtain
\begin{equation}
 \varliminf_{n\to\infty}\E[\scal{w_n-\overline{w}}{\nabla\smo(w_n)-\nabla\smo(\overline{w})}] = 0
\quad \Rightarrow  \varliminf_{n\to\infty}\E[\|\nabla\smo(w_n)-\nabla\smo(\overline{w}) \|^2] =0,
\end{equation}
using again the cocoercivity of $\nabla \smo$.

\ref{p:1ii} We directly get from \eqref{e:cons} that $\sum_{n\in\NN^*} \lambda_n \|u_n\|^2 <+\infty$.

Since $\nabla\smo$ is Lipschitz-continuous, 
and $(\E[\|w_n-\overline{w}\|^2])_{n\in\NN^*}$ is convergent by \ref{p:1i},
there exists $M\in\,]0,+\infty[$ such that
\begin{alignat}{2}
(\forall n\in\NN^{*})\quad
 \E[\scal{w_n-\overline{w}}{ \nabla\smo(w_n)-\nabla\smo(\overline{w})}] \leq \beta 
\E[\|w_n-\overline{w}\|^2] \leq M < +\infty.
\end{alignat}
Hence, we derive from \eqref{e:est4} and \eqref{e:suma} that
\begin{equation}
\label{e:abc1}
 \sum_{n\in\NN^{*}} \lambda_n\gamma^{2}_n 
\E[\|\mathrm{G}_n-\nabla\smo(\overline{w})\|^2] < +\infty.
\end{equation}
 Now, recalling the definition of $u_n$ in \eqref{eq:un}, using \eqref{e:abc1} and  \eqref{e:cons}, we obtain
\begin{equation}
 \sum_{n\in\NN^{*}}\lambda_n\E[\|w_n-y_n\|^2] \leq 2\sum_{n\in\NN} \lambda_n
\E[\|u_n\|^2]
+ 2\sum_{n\in\NN^{*}} \lambda_n\gamma^{2}_n 
\E[\|\mathrm{G}_n-\nabla\smo(\overline{w})\|^2] < +\infty.
\end{equation}
\end{proof}

\begin{proof}[of Theorem \ref{t:2}]
 Since $\smo$ is uniformly convex at $\overline{w}$, 
 there exists  
$\phi\colon\left[0,+\infty\right[\to \left[0,+\infty\right[$ increasing and  
vanishing only at $0$ such that 
\begin{equation}
 \scal{\nabla\smo(w_n)-\nabla\smo(\overline{w})}{w_n-\overline{w}} \geq \phi(\|w_n-\overline{w}\|).
\end{equation}
Therefore, we derive from Proposition \ref{p:1} \ref{p:1iii} that $\sum_{n\in\NN^{*}} \lambda_n\gamma_n\E[\phi(\|w_n-\overline{w}\|)]  < \infty,$
 and hence 
\begin{equation}
 \sum_{n\in\NN^{*}} 
\lambda_n\gamma_n\phi(\|w_n-\overline{w}\|)  < \infty\quad \text{a.s.}
\end{equation}
Since  $(\lambda_n\gamma_n)_{n\in\NN^{*}}$ is not summable,
we have $\varliminf\phi(\|w_n-\overline{w}\|) =0$ a.s. 
Consequently, there exists a subsequence $(k_n)_{n\in\NN^{*}}$ 
such that $ \phi(\|w_{k_n}-\overline{w}\|)\to0$ a.s,
which implies that $\|w_{k_n}-\overline{w}\| \to 0$ a.s. 
In view of Proposition \ref{p:1}\ref{t:2i}, we get $w_n\to \overline{w}$ a.s.
\end{proof}
\begin{proof}[of Theorem \ref{t:3}]
By Proposition \ref{p:1}\ref{p:1i}, 
$(\|w_n-\overline{w}\|^2)_{n\in\NN^*}$ converges to an integrable random variable, hence it is uniformly
 bounded. Moreover,
$\varliminf \E[\|\nabla\smo(w_n)-\nabla\smo(\overline{w})\|^2]= 0$, and hence there exists 
a subsequence $(k_n)_{n\in\NN^{*}}$ such that $\lim_{n\to\infty}\E[\|\nabla\smo(w_{k_n})-\nabla\smo(\overline{w})\|^2]= 0.$
Thus, there
exists a subsequence $(p_{_n})_{n\in\NN^{*}}$ of $(k_n)_{n\in\NN^{*}}$ such that 
\begin{equation}
\label{e:gen1}
 \|\nabla\smo(w_{p_n})-\nabla\smo(\overline{w})\|^2 \to 0 
\quad \text{a.s.} 
\end{equation}
Let $\overline{z}$ be a weak cluster point of $(w_{p_n})_{n\in\NN^{*}}$, 
then there exists a subsequence 
$(w_{q_{p_n}})_{n\in\NN^{*}}$ 
 such that for almost all $\omega$,
$ w_{q_{p_n}}(\omega)\weakly \overline{z}(\omega) $.
Since $\nabla\smo$ is weakly continuous, for almost all $\omega$,
 $\nabla\smo (w_{q_{p_n}}(\omega)) \weakly \nabla\smo(\overline{z}(\omega))$. 
Therefore,  for almost every $\omega$, by \eqref{e:gen1}, $\nabla\smo(\overline{w}) = \nabla\smo(\overline{z}(\omega))$,
and hence
$$\scal{\nabla\smo(\overline{z}(\omega))-\nabla\smo(\overline{w})}{\overline{z}(\omega)-\overline{w}}  =0.$$ 
Since $\smo$ is strictly convex, $\nabla \smo$ is strictly monotone,
 we obtain $\overline{w} = \overline{z}(\omega)$. 
This shows that $w_{q_{p_n}} \weakly \overline{w}$ a.s.
\end{proof}

\begin{proof}[of Remark \ref{r:3}] 
Let $w$ be a weak cluster point of $(w_n)_{n\in\NN^{*}}$,
i.e., there exists a subsequence $(w_{k_n})_{n\in\NN^{*}}$
such that $w_{k_n} \weakly w$ a.s. Since $T = \smo +\nsm$ is convex and lower semicontinous, 
it is weakly lower semicontinous, hence
\begin{equation}
T(w)\leq \varliminf T(w_{k_n}) = \inf T,
\end{equation}
which shows that $w\in\Argmin T$ a.s.  
We therefore conclude that 
$(w_n)_{n\in\NN^{*}}$ converges weakly to an optimal solution a.s.
\end{proof}

\section{Numerical experiments}\label{s:exp}
In this section we first present  numerical experiments  aimed at studying the computational
performance of the SPG algorithm (see Algorithm  \ref{a:main}), with respect to the step-size,
the strong convexity constant, and the noise level. Then we compare the proposed method with 
other state-of-the-art  stochastic first order methods: an accelerated stochastic proximal gradient 
method, called SAGE \cite[Theorem 2]{Hu} and the FOBOS algorithm \cite{Duchi09}. 
\subsection{Properties of SPG}
In order to study the behavior of the SPG algorithm with respect to the relevant parameters of the 
optimization problem, we focus on a toy example, where the exact solution is known.
More specifically, we consider the following minimization problem on the real line:
\begin{equation}
\label{e:cyc1}
  \underset{w\in \RR}{\text{minimize}}\; \phi(w) :=\frac{\mu}{2}|w-10|^2 + 0.02|w-10|.
\end{equation}
It is clear that $\phi$ is $\mu$-strongly convex function with $w_{opt}=\argmin\phi = \{10\}$ and the optimal value 
$\overline{\phi} =0$. We consider a stochastic perturbation of the exact gradient of the function 
$\smo = \frac{1}{2}|\cdot-10|^2 $ of the form
\begin{equation}\label{eq:stgr}
 \mathrm{G}_n = \nabla\smo(w_n) + s_n, 
\end{equation}
where $s_n$ is a realization of a Gaussian random variable with $0$ mean and $\sigma^2$  variance. 
We apply SPG one hundred times  for 100 independent realizations of the random
process $(s_n)_{n\in\mathbb{N}^*}$ to problem \eqref{e:cyc1} with $(\forall n\in\NN^*)$ $\lambda_n=1$
and $\gamma_n=C/n$ for some constant $C>0$.  
We evaluate the average performance of SPG over the first 100 iterations for different values of the strong convexity 
parameter $\mu$, and several values of $\sigma$ and $C$, and by measuring $|w_n-10|$.
The results are displayed in Figure  \ref{fig: orbit1pcabc}. As can be seen by visual inspection,
the convergence is faster when $\mu$ is bigger and when the noise variance is smaller. Moreover, 
the constant $C$ in the step-size heavily influence the convergence 
behavior. The latter is a well-known phenomenon in the context of stochastic optimization \cite{Nem09}, 
 
\begin{figure}[!ht]
\centering
\begin{tabular}{ccc}
 \includegraphics[width=4cm]{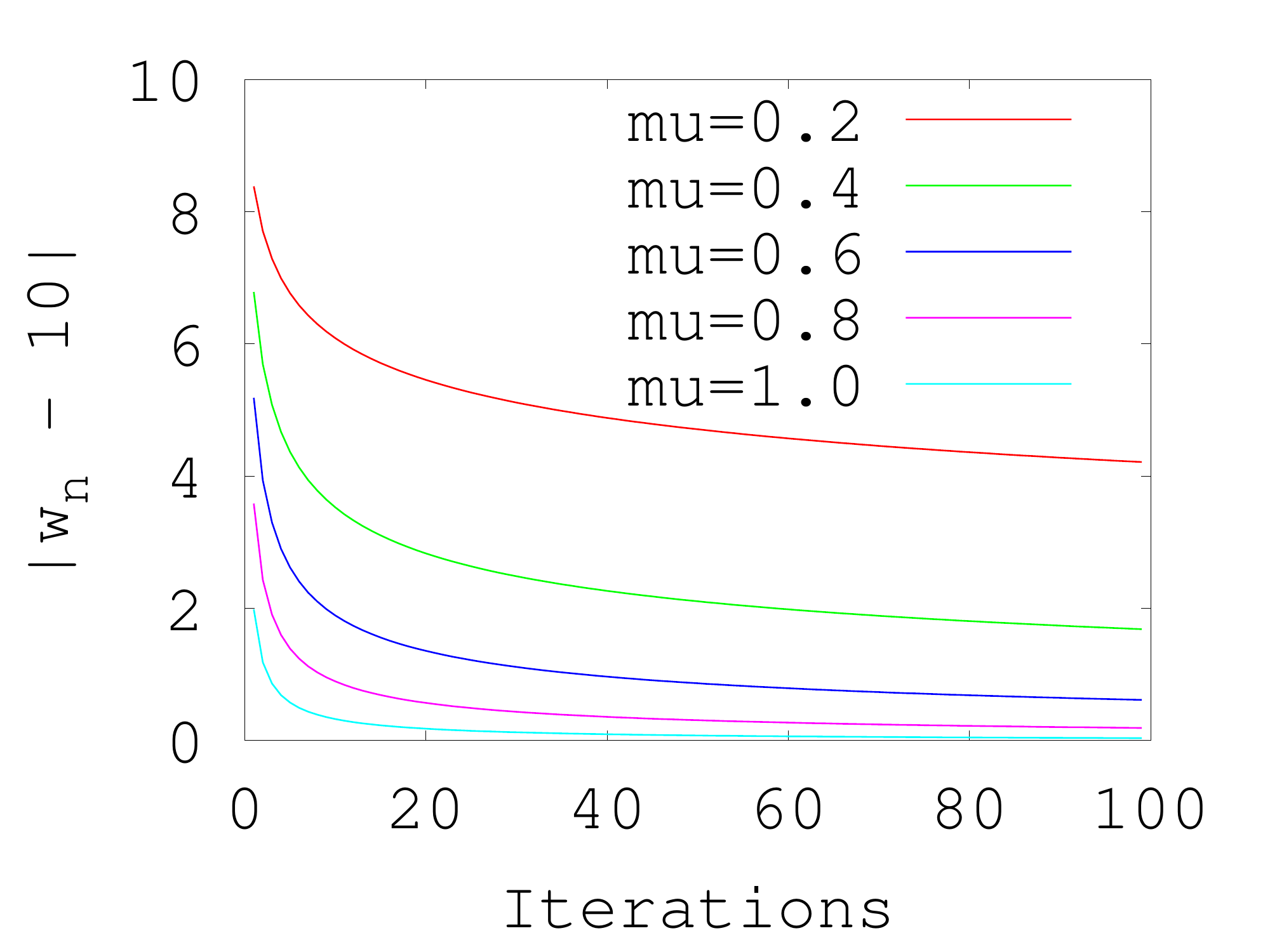}& \includegraphics[width=4cm]{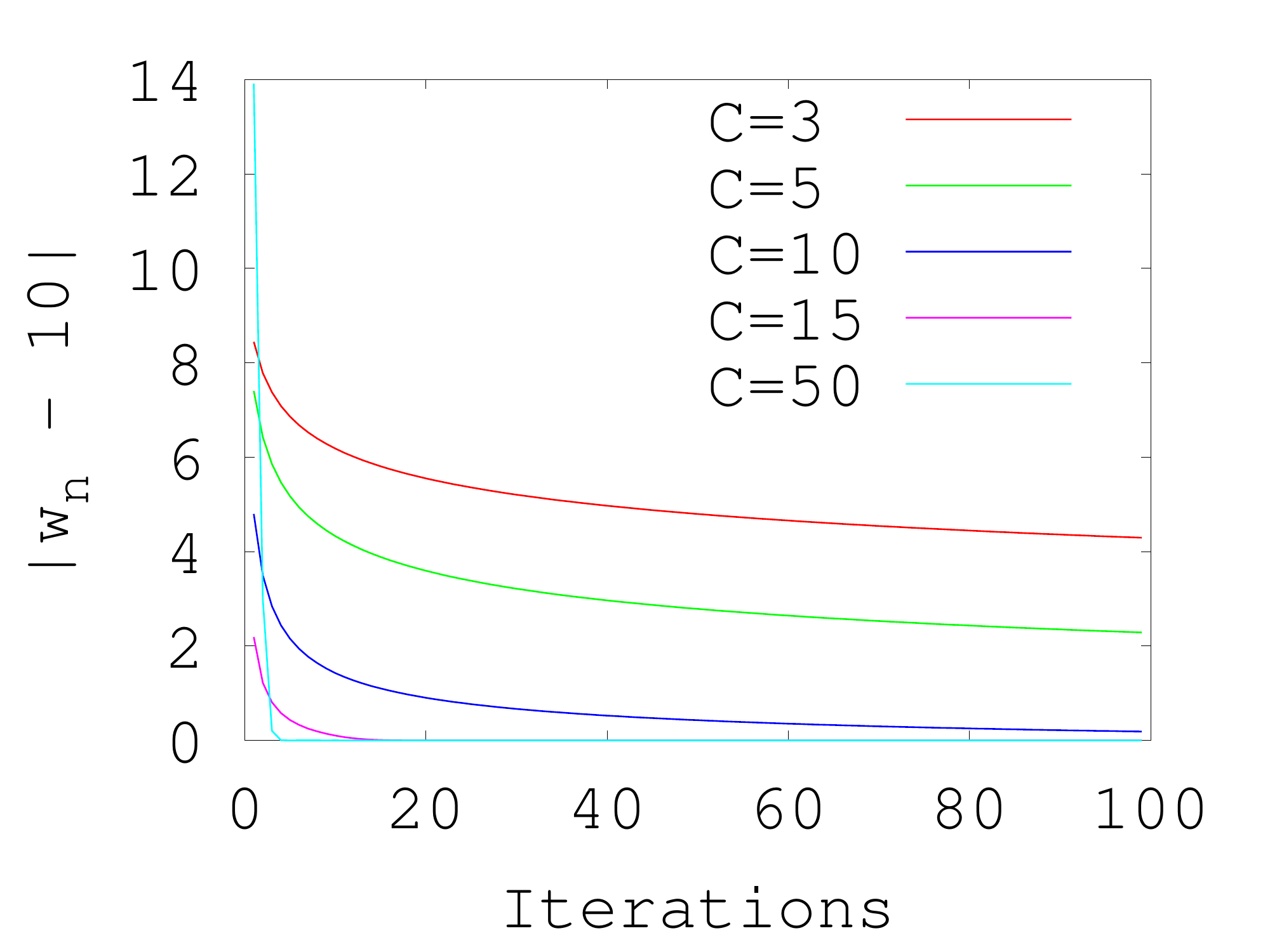}& \includegraphics[width=4cm]{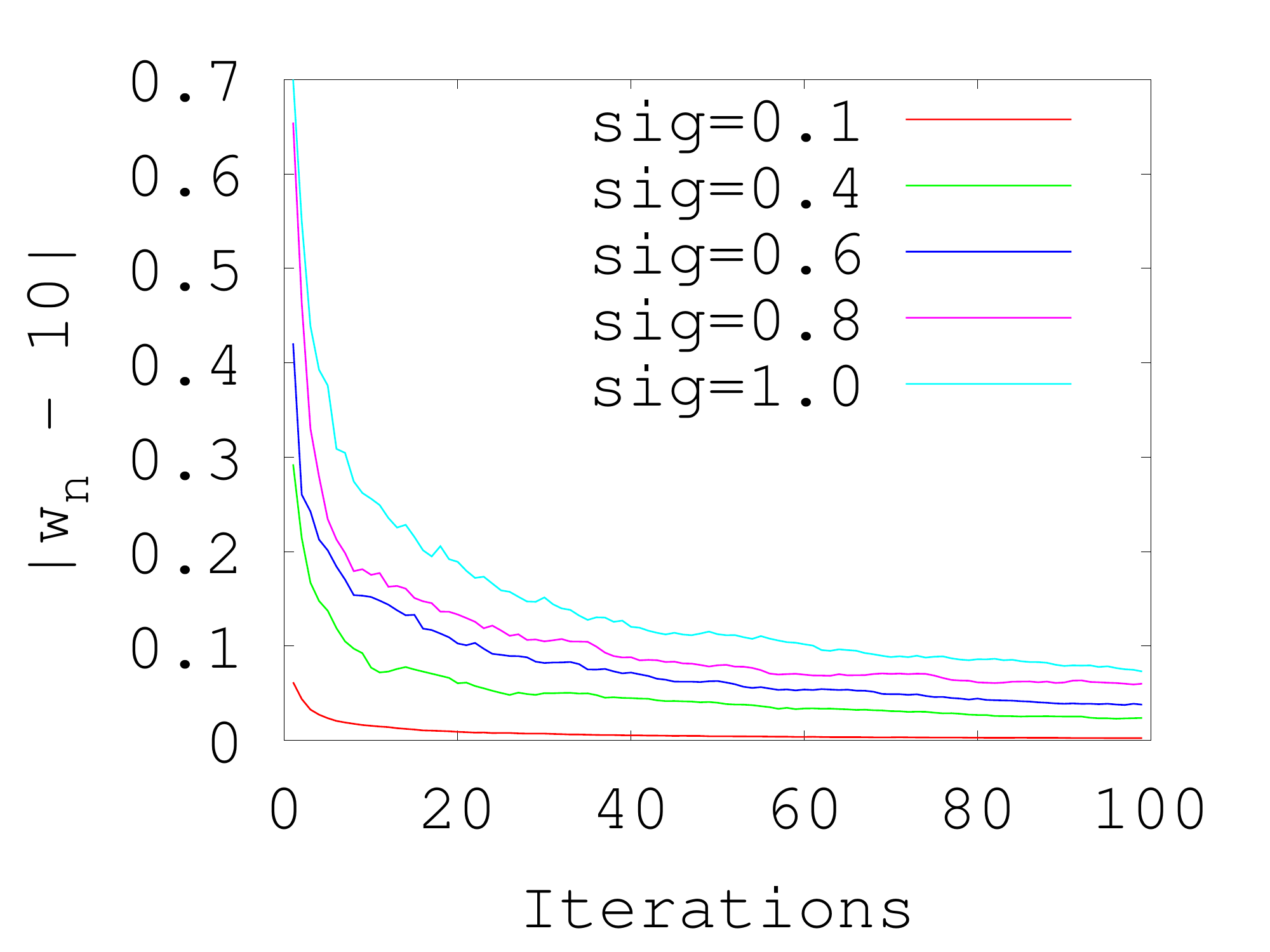}
 \end{tabular}
 \caption{Performance evaluation of  Algorithm \ref{a:main} with respect to different choices of  $\mu$, setting $\gamma_n=0.8/n$ and $\sigma=0.01$ (left), with respect to different choices of $C$, with $\gamma_n=C/n$, $\mu=0.05$ and $\sigma=0.01$ (center), and with respect to different choices of $\sigma$, for $\mu=0.05$  and $\gamma_n=1/n$(right).
 \label{fig: orbit1pcabc} }
 \end{figure}
\subsection{Comparison with other methods}
In this section we compare  SPG  with the SAGE algorithm \cite[Algorithm 1]{Hu} and the FOBOS algorithm in \cite{Duchi09}. 
We note that the main difference between SPG (with  $\lambda_n=1$ for every $n\in\NN^*$) 
and FOBOS is  that the latter takes the average
of the previous iterates. More precisely the sequence generated by the FOBOS iteration  is the following
\begin{equation}\label{eq:sssss}
w^{(av)}_n = \Big(\sum_{k=1}^n \eta_k\Big)^{-1}\sum_{k=1}^n \eta_kw_k, \quad \eta_k = C_1/k,
\end{equation}
where $(w_k)_{k\in\NN^*}$ is the sequence generated by the SPG algorithm. In \cite{Duchi09} it 
is assumed that the gradient of the smooth term is bounded on the whole space. 
In our experiments this assumption is not satisfied, but since  the sequence
 of iterates is bounded, the algorithm can be applied and its convergence is guaranteed.
One advantage of the SAGE algorithm is that it does require any parameter tuning, since it does 
not have  any  free parameter. SPG and FOBOS instead require the choice of the stepsize. 
We check the accuracy of the three algorithms on different elastic net regularized problems 
with respect to the number of iterations, since the cost  per iteration is basically the same for the three 
procedures. 

\subsection{Toy example}
 We first consider the toy example presented in the previous section (see equation \eqref{e:cyc1}),
 where  we set  $\mu=1$. Moreover, we assume that in \eqref{eq:stgr} $s_n$ is a realization of a Gaussian 
 random variable with $0$ mean and $0.1$ variance. We run SPG, SAGE, and FOBOS one hundred times for one hundred 
 independent realizations of the random process $(s_n)_{n\in\NN^*}$.
In SPG, we chose $\gamma_n=1/n$, and $\lambda_n=1$.
 Finally, after testing  the FOBOS algorithm  for different choices of the constant
$C_1$  defining the stepsize,  we got that $\eta_k=1/k$ for every $k\in\NN^*$ gave the best results.
 The behavior of  the sequences $|w_n-10|$  corresponding to the three algorithms  on the first 1000 iterations
 is presented in Figure~\ref{fig: orbit1pc}.  SPG and SAGE have a similar behavior, while FOBOS is 
 slower.
 \begin{figure}[!ht]
   \centering     
 \includegraphics[width=8cm]{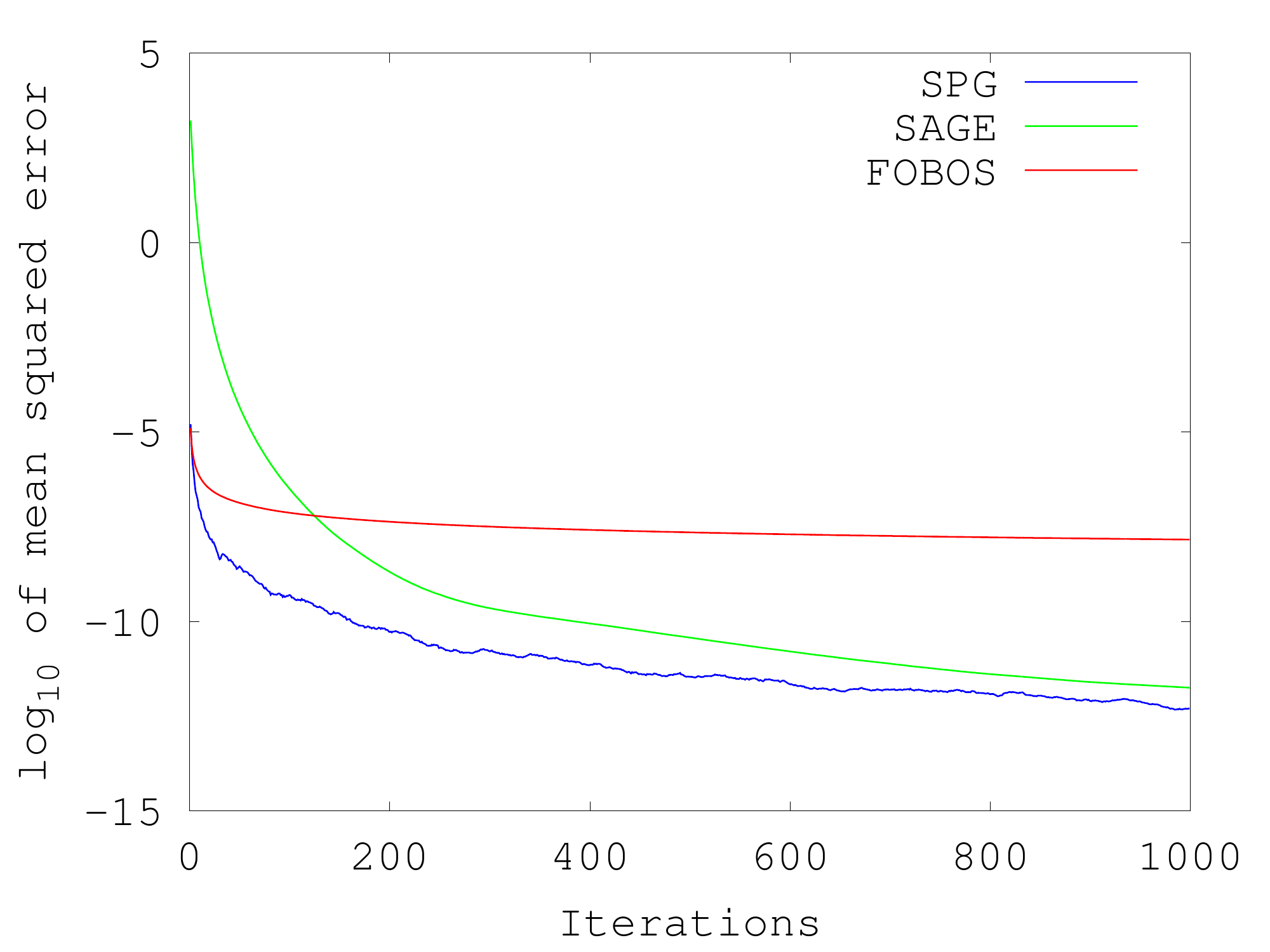}
 \caption{Convergence  of  SPG, SAGE, and FOBOS. The performance of SAGE and SPG is comparable, while
 FOBOS is slower on this example.  
 \label{fig: orbit1pc}}
 \end{figure}
\subsection{Regression problems with random design}
Let $N$ and $p$ be strictly positive integers. 
Concerning the data generation protocol, the input points  $(x_i)_{1\leq i\leq N}$
 are uniformly drawn in the interval $\left[a,b\right]$ (to be specified later in the two cases we consider). 
 For a suitably chosen finite dictionary 
 of real valued functions $(\phi_k)_{1\leq k\leq p}$ 
 defined on $\left[a,b\right]$, the labels are computed using a noise-corrupted regression
 function, namely
\begin{equation}
(\forall i \in \{1,\ldots, N\})\quad y_i =  \sum_{k=1}^{p}\overline{w}_k\phi_k(x_i)+ \epsilon_i,
\end{equation}
where $(\overline{w_k})_{1\leq k\leq p} \in\mathbb{R}^p$ and $\epsilon_i$ is an additive noise $\epsilon_i\sim \mathcal{N}(0,0.3)$.

We will consider two different choices for the  dictionary of functions:  polynominals, i.e. 
$(\forall k\in \{1,\ldots,p\})$ $\phi_k\colon \left[-1,1\right]\to \RR$, $\phi_k(x)=x^{k-1}$
and trigonometric functions, i.e. $p=2q+1$ and  
$(\forall k\in \{1,\ldots,q\})\; \phi_k \colon \left[0,2\pi\right]\to \RR$, $\phi_k(x)=\cos((k-1)x)$ and 
$(\forall k\in \{q+1,\ldots,2q+1\})\; \phi_k \colon \left[0,2\pi\right]\to \RR$, $\phi_k(x)=\sin(kx)$.
The training set  and the regression function for the two examples are presented in  Figure~\ref{fig: orbit1pc1q1d}. 
\begin{figure}[!ht]
   \centering    
   \begin{tabular}{cc} 
 \includegraphics[width=6cm]{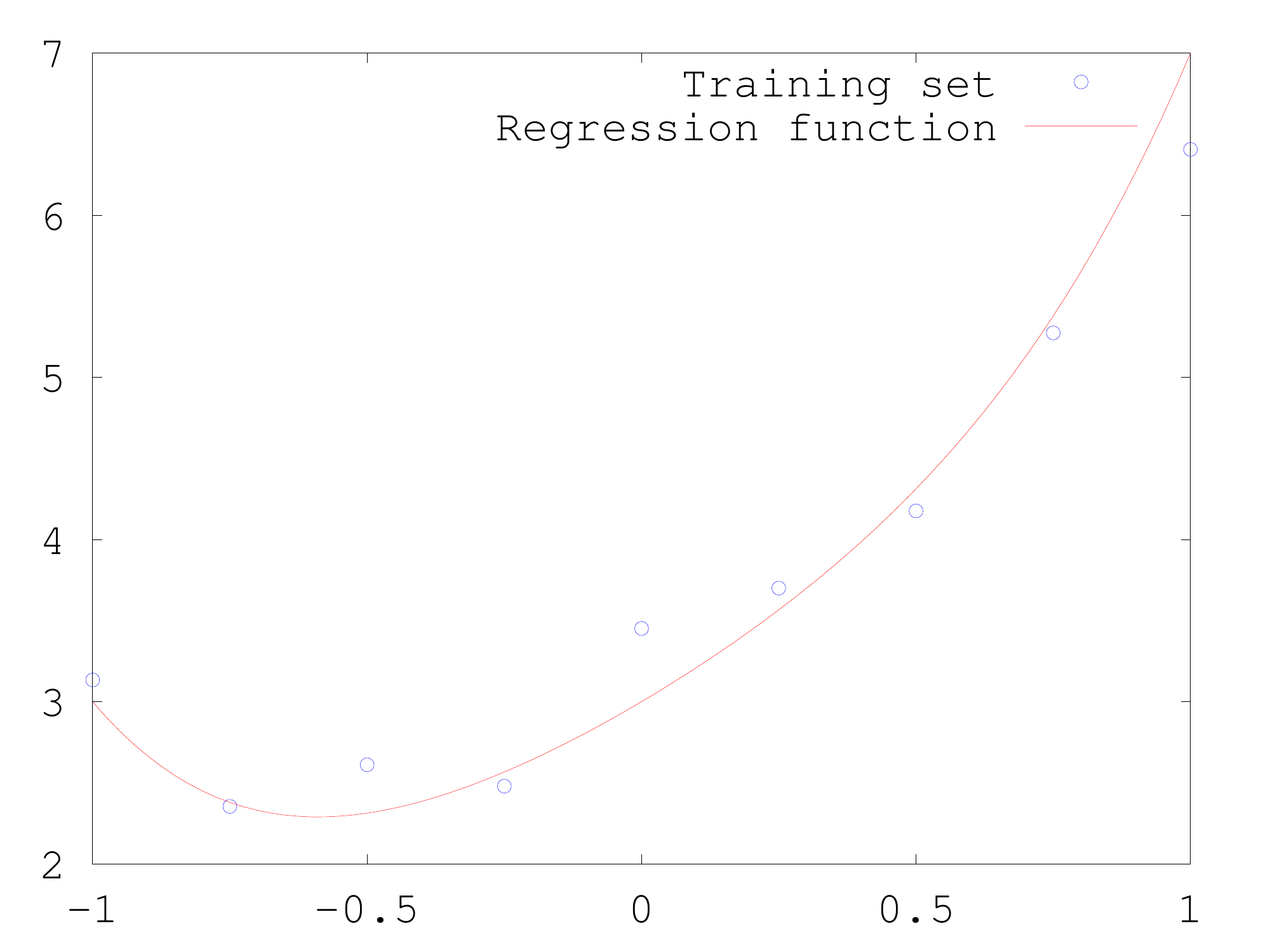}& \includegraphics[width=6cm]{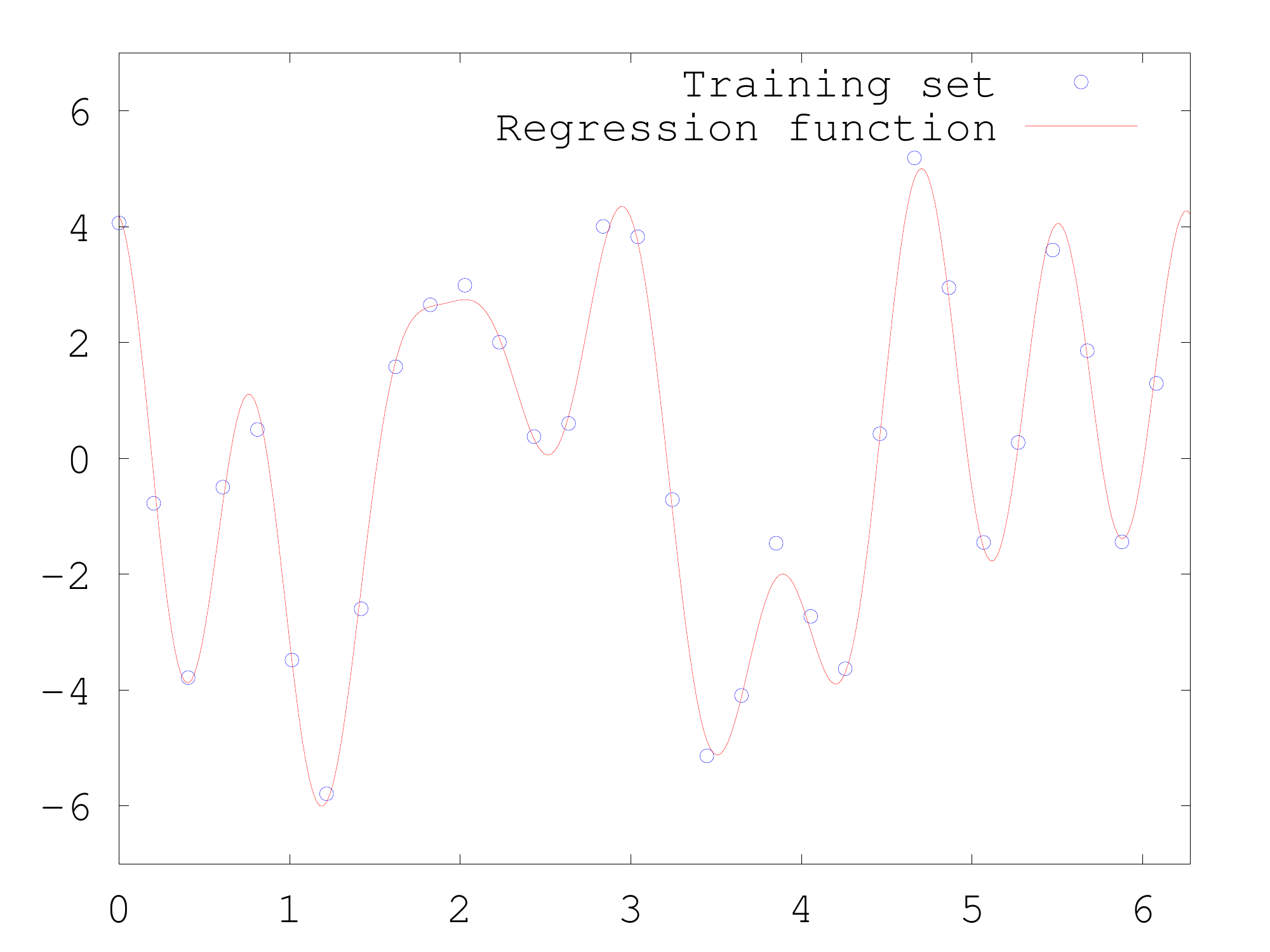}
 \end{tabular}
 \caption{Left: training set and regression function, on the interval $[-1,1]$ 
for the polynomial dictionary with $p=6$. Right: training set and regression function on $[0,2\pi]$
for the Fourier basis with $p=21$.
}
 \label{fig: orbit1pc1q1d} 
 \end{figure}
We estimate $\overline{w}$ by
 solving the following regularized minimization problem 
\begin{equation}
\label{e:ff1}
 \underset{ (w_k)_{1\leq k\leq p}\in\RR^p}{\text{minimize}} \;\frac{1}{2N}
\sum_{i=1}^N\Big(y_i -\sum_{k=1}^{p}w_k\phi_k(x_i)\Big)^2
+ \frac{1}{2}\sum_{k=1}^{p} (\mu|w_k|^2+ \omega|w_k|  \big),
\end{equation}
where $\mu$ and $\omega$ are strictly positive parameters.   
Problem  \eqref{e:ff1} is a special case of  Example \ref{Ex:SumOfFunctions}, and hence it can be 
solved by using SPG, SAGE, and FOBOS in an incremental fashion. 
For the polynomial dictionary, we set 
 \begin{alignat}{2} 
&p=6, \quad N=9,\quad \gamma_n = 15/(n+100), \quad  \eta_n = 15/(n+100),\quad 
\mu = 0.1,\quad \omega = 0.01, \notag\\
&\overline{w}=[3,2,1,0,1,0].
\end{alignat}
For the trigonometric dictionary, we set 
 \begin{alignat}{2} 
&p=21,\quad N=32,\quad \gamma_n=\eta_n = 10/(n+100),\quad 
\mu = 0.01,\quad \omega= 0.01, \notag\\
&\overline{w}=[0,0.2,0,0.5,1,-1,0,1,2,0.5,0,0,-0.1,-2.5,1,0,0,-1,0.9,-0.5,0].
\end{alignat}

The resulting regression functions using the three algorithms are shown in  Figure \ref{fig: orbit1pc1z1}.
As can be seen from visual inspection, the three methods provide almost undistinguishable solutions. 
\begin{figure}[!ht]
   \centering     
   \begin{tabular}{cc}
 \includegraphics[width=6cm]{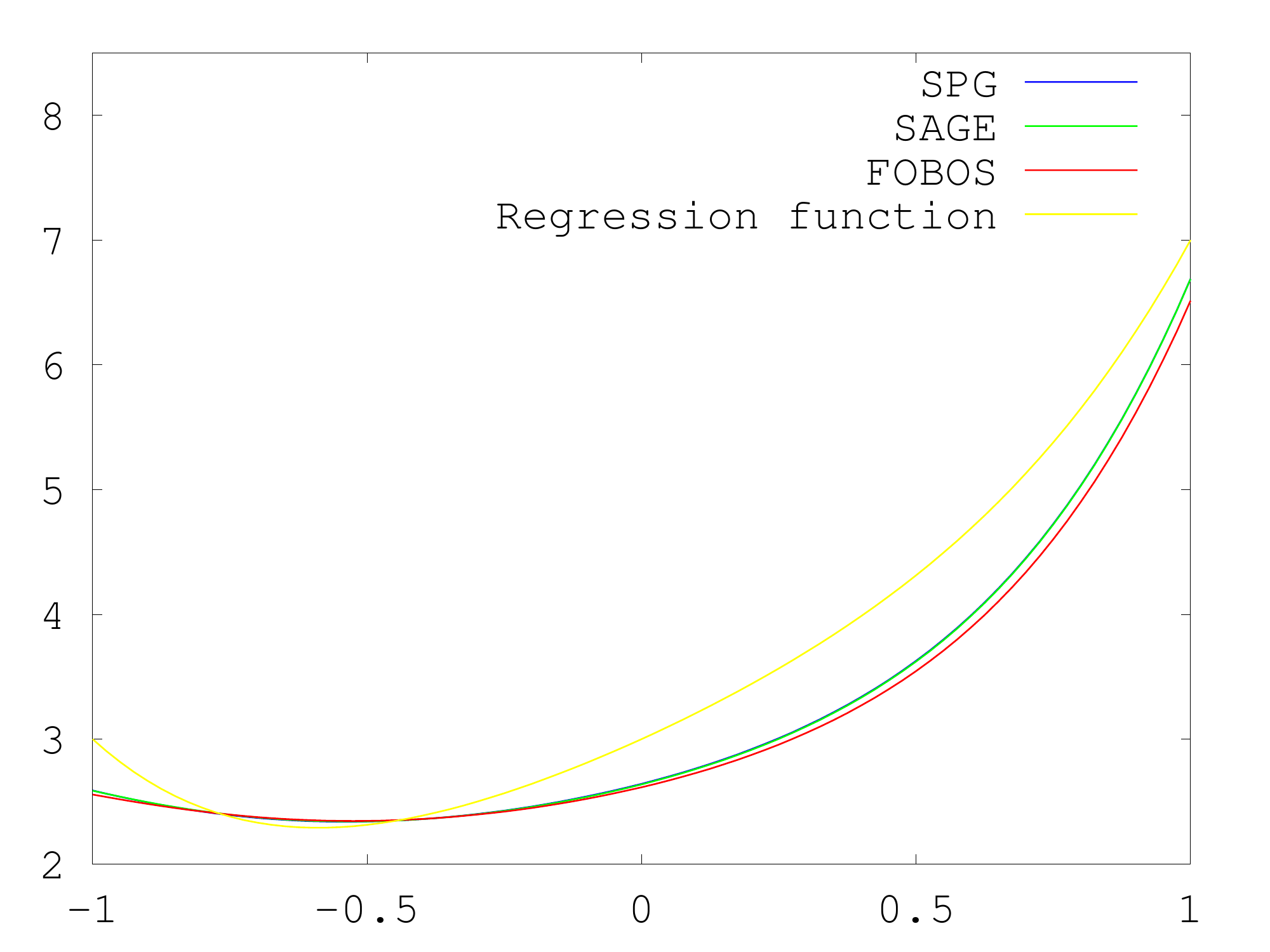}& \includegraphics[width=6cm]{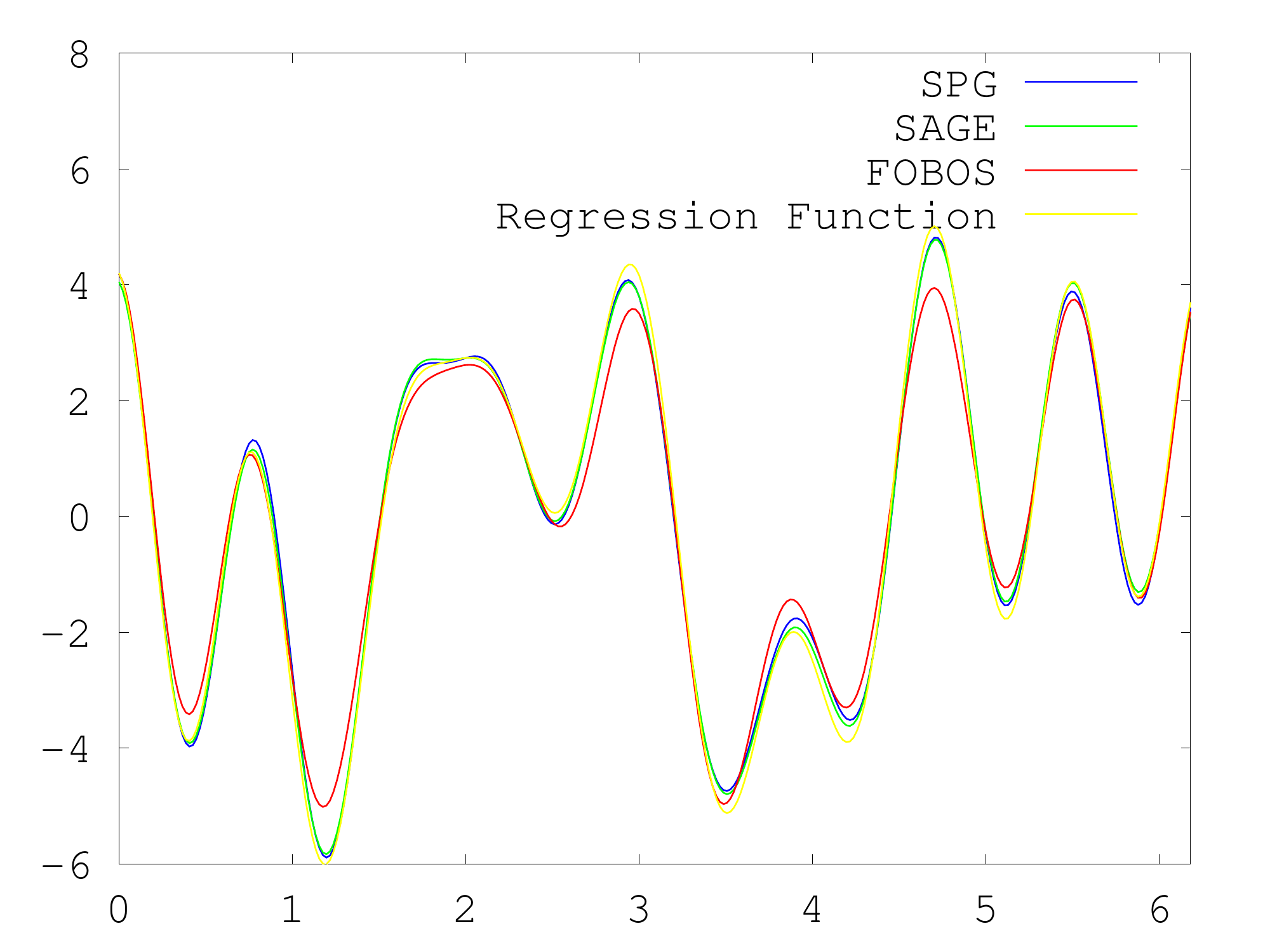}
 \end{tabular}
 \caption{Regression functions obtained using SPG, SAGE, and FOBOS with polynomial dictionary (left) and trigonometric dictionary (right).
 \label{fig: orbit1pc1z1} }
 \end{figure}

Finally, we computed an approximate solution of \eqref{e:ff1} by running the forward-backward 
splitting method in \cite{siam05}  for 50000 iterations. 
The convergence of  the iterations to the solution of \eqref{e:ff1} is displayed  in  Figure 
\ref{fig: orbit1pc1z1a}. On the regression problem with the polynomial dictionary, SAGE is performing the best, while
on the trigonometric dictionary, SPG is the fastest. The oscillating behavior is mitigated by the averaging procedure at the expanses of a slower convergence rate,
as the more regular behaviour of FOBOS clearly shows.
\begin{figure}[!ht]
   \centering     
   \begin{tabular}{cc}
 \includegraphics[width=6cm]{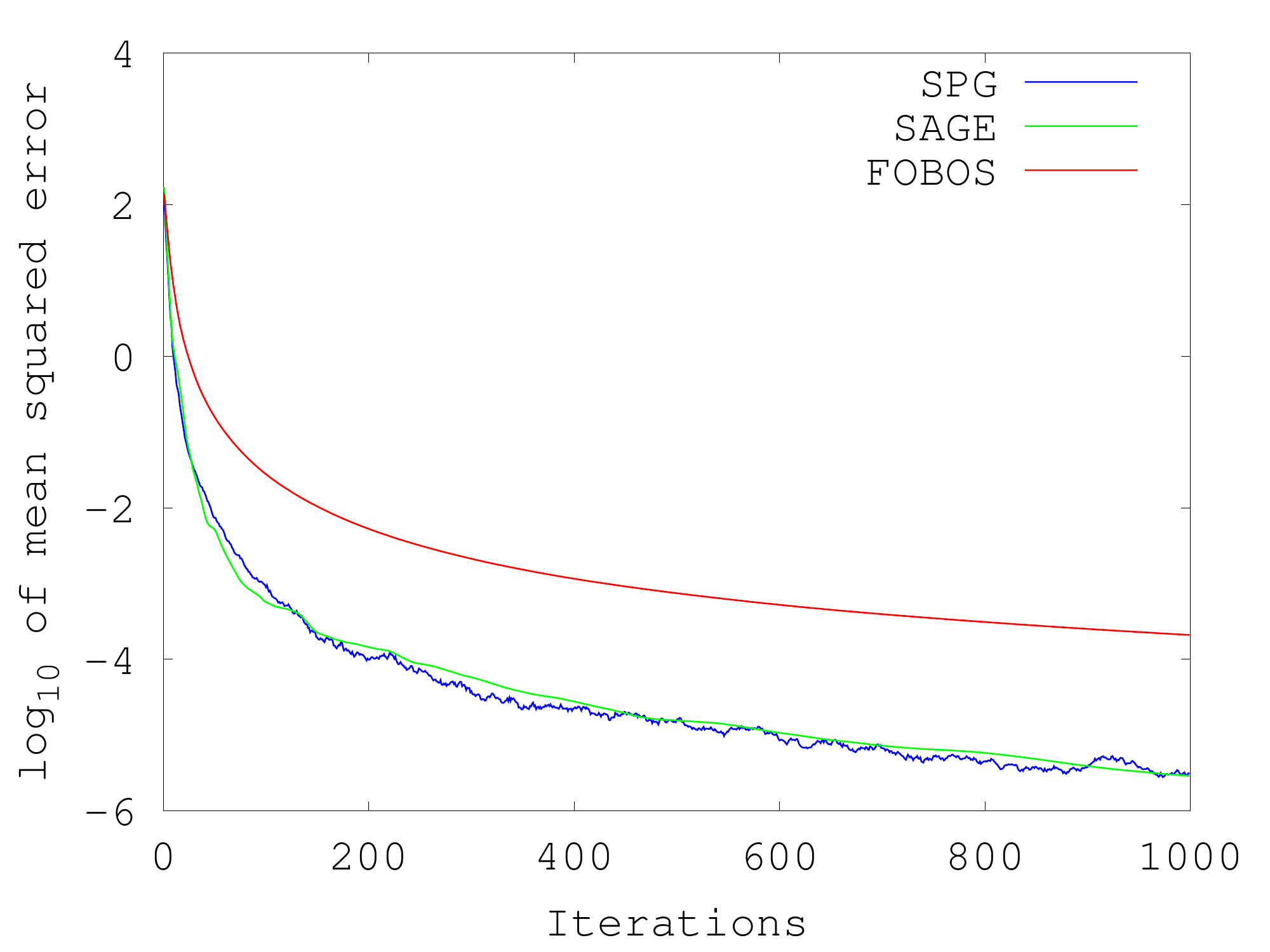}& \includegraphics[width=6cm]{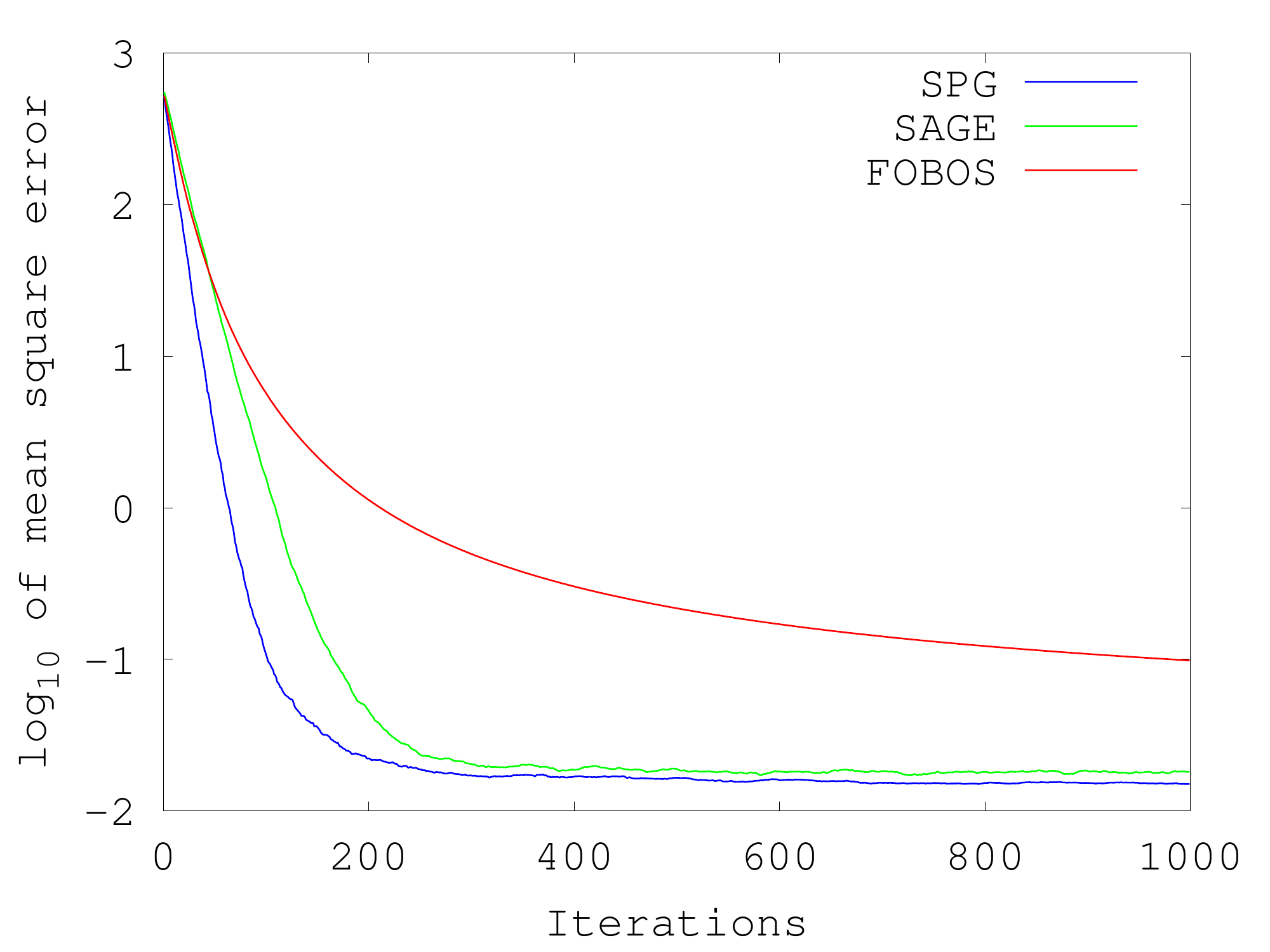}
 \end{tabular}
 \caption{The convergence of the iterations to the optimal solution of  \eqref{e:ff1} with polynomial dictionary (left)
 and with trigonometric dictionary (right).
 \label{fig: orbit1pc1z1a} }
 \end{figure}
\subsection{Deconvolution problems}
As a last experiment, we  focus on the problem of recovering an ideal signal $\overline{w}$ 
from a noisy observation of the form 
\begin{equation}\label{eq:den}
 \RR^{1024}\ni y = h*\overline{w} + s,
\end{equation}
where $s\sim \mathcal{N}(0,0.06)$ and $h$ is a Gaussian kernel.
\begin{figure}[!ht]
   \centering     
 \includegraphics[width=8cm]{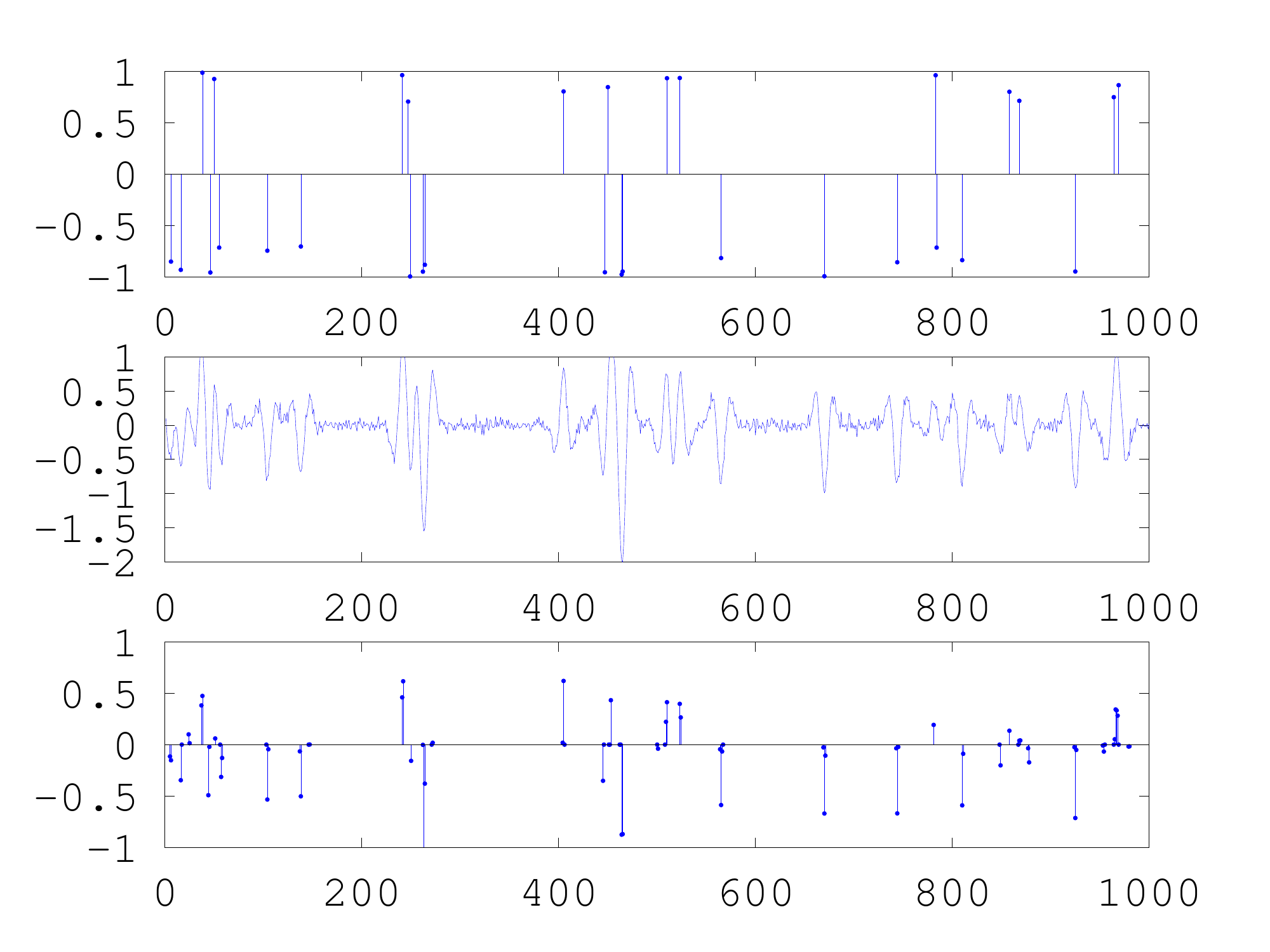}
 \caption{The ideal signal (top), the noisy signal (middle), and the restored signal (bottom) by SPG. 
 \label{fig: orbit1pc5} }
 \end{figure}
To find an approximation of the ideal signal, we solve the following variational 
problem 
\begin{equation}
\label{e:cyc2}
  \underset{w\in  \RR^{1024}}{\text{minimize}}\; T(w),\quad\quad T(w)=\frac{1}{2}\|y-h*w\|^2 
+ \|w\|_1 + \frac{0.02}{2}  \|w\|^{2}_2.
\end{equation}
An approximation $\overline{w}$ of the exact solution is found by running the forward-backward splitting method in \cite{siam05} 
for $10000$ iterations. Then,
we run SPG, SAGE, and FOBOS with the same initialization, for $5000$ iterations using at the $n$-th iteration a stochastic gradient of the form
\begin{equation}
 \mathrm{G}_n = \nabla\smo(w_n) + s_n, 
\end{equation}
where $s_n\sim \mathcal{N}(0,0.01)$ and  $\smo(w) = \frac{1}{2}\|y-h*w\|^2 + \frac{0.02}{2}  \|w\|^{2}_2 $. 
FOBOS is run with $\eta_n = 3/(n+100)$. This is not the theoretically optimal choice, but gave better
results in practice. 
In SPG we set $\lambda_n=1$ and $\gamma_n=3/(n+100)$.
Convergence of $(\|w_n-\overline{w}\|)_{n\in\NN^*}$ for the three algorithms is presented in Figure ~\ref{fig: orbit1pc1}.
In this case SAGE is the fastest, and SPG shows slightly worse convergence. FOBOS is again slower.  
\begin{figure}[!ht]
   \centering     
 \includegraphics[width=8cm]{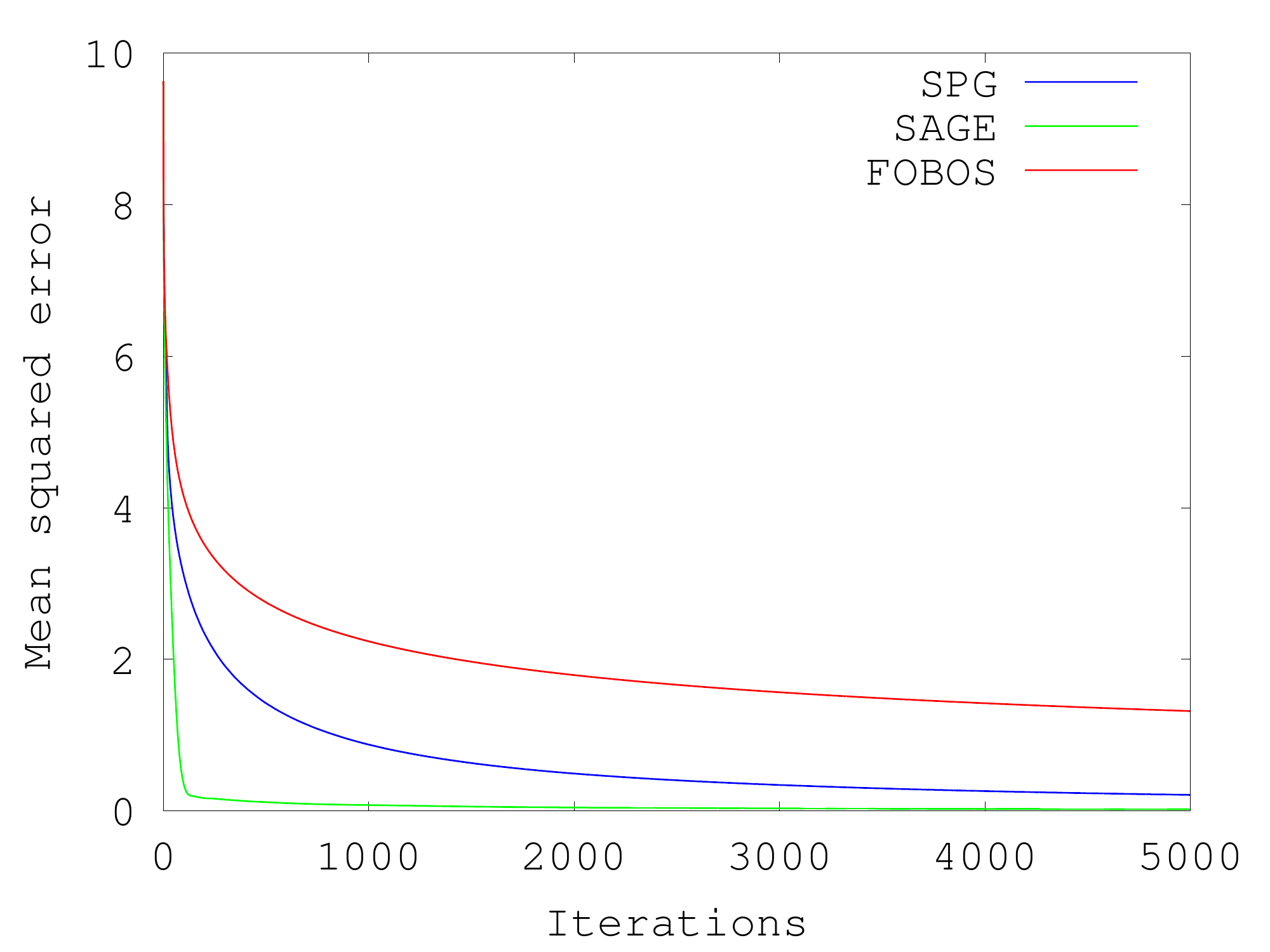}
 \caption{Convergence of the iterates for SPG, SAGE, and FOBOS with starting point 0.
 \label{fig: orbit1pc1} }
 \end{figure}

Finally, we address the problem of the iterations' sparsity. We generate the data according to the model in \eqref{eq:den},
starting from an original signal with $993$ zero components.
In Figure \ref{fig: orbit1pc1a}  we display the number of zero components
of the iterates. As it can be readily seen by visual inspection, after few iteratons both SAGE ans SPG
generate sparse iterations. On this example this does not hold for the FOBOS algorithm, for which
the sparsity of the iterates is  a decreasing function of the number of iterations.  
The number of zero components of  the last iterate of SPG, SAGE, and FOBOS is 937, 937, and 438, respectively.
\begin{figure}[!ht]
   \centering     
 \includegraphics[width=8cm]{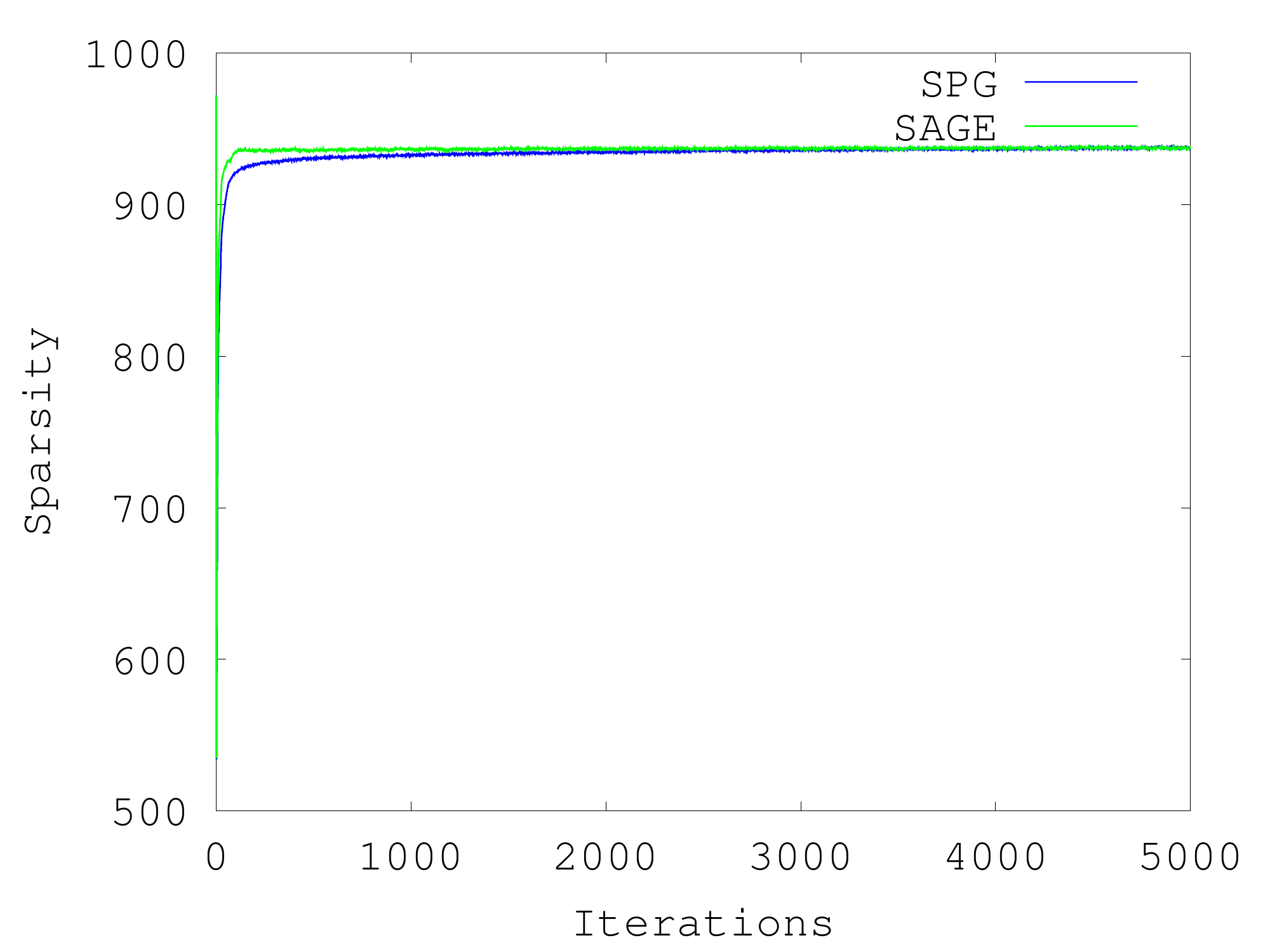}
 \caption{Number of zero components of the vector $(w_{n}-\overline{w})_{n\in\NN^*}$ with the same initial point $0$ for
 SPG and SAGE. 
 \label{fig: orbit1pc1a} }
 \end{figure}

\section{Conclusion}
In this paper we proposed and studied  a stochastic  approach to the problem of minimizing a  strongly convex non-smooth function. In particular, we have considered the case of composite minimization (the objective function is the sum of a smooth and a convex term) and proposed a stochastic extension of proximal splitting methods which have become widely popular in a deterministic setting. These latter approaches are based on recursively  computing the gradient 
of the smooth term and then applying the proximity operator defined by the convex term. 
The starting point of the paper is considering the case where only stochastic estimates of the 
gradient are available. This latter situation is relevant in online approaches to learning and more generally in stochastic and incremental optimization. 

The main contributions of this paper are to provide  convergence rates in expectation and to establish almost sure convergence of the proposed stochastic proximal gradient method.
A further contribution regards our proving techniques, which differ from many previous approaches based on regret analysis and online-to batch conversion. Indeed, our approach is based on extending convex optimization techniques from the deterministic to the stochastic setting.  Such extensions are interesting in their own right and could lead to further applications.
An outcome of our analysis is that,  unlike in the deterministic case, in the stochastic case acceleration does not yield any improvement in the rates.

The analysis in the paper suggests a few venues for future work.  A main one is the relaxation of  the strong convexity assumption, considering in particular objective functions  which are only convex; steps in this direction have been taken in \cite{ComPes14}.  Deriving high probability, rather than expectation bounds, would also be interesting.

{\small {\bf Acknowledgments.} This material is based upon work supported by the Center for Brains, Minds and Machines (CBMM), 
funded by NSF STC award CCF-1231216.
L. R. acknowledges the financial support of the Italian Ministry of Education, University and Research FIRB project RBFR12M3AC. 
S. V. is member of the Gruppo Nazionale per
l'Analisi Matematica, la Probabilit\`a e le loro Applicazioni (GNAMPA)
of the Istituto Nazionale di Alta Matematica (INdAM). }

\appendix
\section{Proofs of auxiliary results}\label{s:bg}

Here we prove the statements about random quasi-Fej\'er sequences and an upper bound
on a numerical sequence satisfying the assumptions of Lemma \ref{l:ics}.
\begin{proof}[of Proposition \ref{p:fejer}]
 It follows from \eqref{e:fejer} that 
\begin{equation}
\label{e:refejer}
(\forall n\in\NN^{*})(\forall w\in S)\quad \E[\|w_{n+1}-w\|^2] 
\leq \E[\|w_n-w\|^2] + \varepsilon_n \,.
\end{equation}

\ref{p:fejeri}: Since the sequence $(\varepsilon_n)_{n\in\NN^*}$ is summable and
$\E[\|w-w_1\|^2]$ is finite, we 
derive from \eqref{e:refejer} that $(\E[\|w_n-w\|^2])_{n\in\NN^*}$ is a real positive quasi-Fej\'er
sequence, and therefore it converges to some
$\zeta_{w}\in \RR$ by \cite[Lemma 3.1]{Com01}. Set 
\begin{equation}
(\forall n\in\NN^*)\quad r_n = \|w_n-w\|^2 + \sum_{k=n}^{\infty}\varepsilon_n.
\end{equation}
Then, it follows from \eqref{e:fejer} that
\begin{alignat}{2}
(\forall n\in\NN^*)\quad \E[r_{n+1}|\mathcal{A}_n] &=  \E[\|w_{n+1}-w\|^2  |\mathcal{A}_n] + \sum_{k=n+1}^{\infty}\varepsilon_n\notag\\
&\leq \|w_n-w\|^2 +  \sum_{k=n}^{\infty}\varepsilon_n\notag\\
&= r_n.
\end{alignat}
Therefore $(r_n)_{n\in\NN}$ is a (real) supermartingale. Since $ \sup_n \E[\min\{r_n, 0\}] < +\infty$
 by \eqref{e:refejer}, $r_n$ converges a.s to 
an integrable random variable \cite[Theorem 9.4]{Met82}, that we denote by $\xi_{w}$.

\ref{p:fejerii}\&\ref{p:fejeriii}: Follow directly by \ref{p:fejeri}.
\end{proof}

\begin{proof}[of Lemma \ref{l:ocs}]  
Note that, for every $m\in\NN^*$, $n\in\NN^*$, $m\leq n$:
\begin{equation} \label{eq:varineq}
 \sum_{k=m}^n k^{-\alpha} \geq \varphi_{1-\alpha}(n+1) -\varphi_{1-\alpha} (m),
\end{equation}
where $\varphi_{1-\alpha}$ is defined by \eqref{e:bach1}.
Since all terms in \eqref{e:iter} are positive for $n\geq n_0$, 
by applying the recursion $n-n_0$ times we have
 \begin{equation}\label{eq:rec}
 s_{n+1} \leq s_{n_0}\prod_{k=n_0}^n(1-\eta_k)  
+\tau \sum_{k=n_0}^n \prod_{i=k+1}^n(1-\eta_i)\eta_{k}^2. 
\end{equation}
Let us  estimate the first term in the right hand side of \eqref{eq:rec}. 
Since  $1-x\leq \exp(-x)$  for every $x\in\mathbb{R}$, from  \eqref{eq:varineq}, we derive
\begin{alignat}{2}
\label{e:conca}
\nonumber s_{n_0}\prod_{k=n_0}^n\left(1-\eta_k\right)  
= s_{n_0} \prod_{k=n_0}^n\Big(1-\frac{c}{k^{\alpha}}\Big)
&\leq s_{n_0} \exp\left(-c\sum_{k=n_0}^n k^{-\alpha} \right)\\ & \leq\begin{cases}
    s_{n_0}\big(\frac{n_0}{n+1}\big)^{c}& \text{if $\alpha = 1$},\\
    \\[-2ex]
s_{n_0} \exp\Big(\frac{c}{1-\alpha}(n_{0}^{1-\alpha} - (n+1)^{1-\alpha}) \Big)
& \text{if $0< \alpha < 1$.}
   \end{cases}
\end{alignat}
To estimate the second term in the right hand side of \eqref{eq:rec}, let us first consider the case $\alpha < 1$,
let $m\in\mathbb{N}\setminus\{0\}$ such that $n_0\leq n/2 \leq  m+1 \leq (n+1)/2$. We have 
\begin{alignat}{2}
 \sum_{k= n_0}^n \prod_{i=k+1}^n&(1-\eta_i)\eta_{k}^2 =\sum_{k= n_0}^m \prod_{i=k+1}^n(1-\eta_i)\eta_{k}^2 +\sum_{k=m+1}^n \prod_{i=k+1}^n(1-\eta_i)\eta_{k}^2 \notag\\
&\leq \exp\big(-\sum_{i=m+1}^n\eta_i\big)\sum_{k=n_0}^m\eta_{k}^{2} + {\eta_m} \sum_{k=m+1}^n \left(\prod_{i=k+1}^n(1-\eta_i)-\prod_{i=k}^n(1-\eta_i)\right)\notag\\ 
&=\exp\big(-\sum_{i=m+1}^n\eta_i\big)\sum_{k=n_0}^m\eta_{k}^{2} + {\eta_m}  \left(1-\prod_{i=m+1}^n(1-\eta_i)\right)\notag\\ 
&\leq \exp\big(-\sum_{i=m+1}^n\eta_i\big)\sum_{k=n_0}^m\eta_{k}^{2} +\eta_m\notag\\
&\leq 
c^2\exp\Big(\frac{c}{1-\alpha}( (m+1)^{1-\alpha} - (n+1)^{1-\alpha}) \Big)
\varphi_{1-2\alpha}(n) + \eta_m\\
\label{eq:rhss}&\leq 
c^2\exp\Big(\frac{-ct(n+1)^{1-\alpha}}{1-\alpha} \Big)
\varphi_{1-2\alpha}(n) + \frac{2^{\alpha} c }{\mu(n-2)^{\alpha}}.
\end{alignat}
Hence, combining \eqref{e:conca} and \eqref{eq:rhss}, for $\alpha\in\left]0,1\right[$ we get
\begin{alignat}{2}
 \quad 
s_{n+1} &\leq \Big(\tau c^2 \varphi_{1-2\alpha}(n)
 + s_{n_0}\exp\Big(\frac{cn_{0}^{1-\alpha}}{1-\alpha}\Big) \Big)
\exp\Big(\frac{-ct(n+1)^{1-\alpha}}{1-\alpha} \Big)
+ \frac{\tau 2^{\alpha}c}{(n-2)^{\alpha}}.
\end{alignat}
We next estimate the second term  in the right hand side of \eqref{eq:rec} in the case $\alpha =1$. We have
\begin{alignat}{2}
 \sum_{k= n_0}^n \prod_{i=k+1}^n(1-\eta_i)\eta_{k}^2 
&=  \frac{c^2}{(n+1)^{c}}\Big(1+ \frac{1}{n_0}\Big)^{c} 
\sum_{k= n_0}^{n}\frac{1}{k^{2-c}} 
\leq   \frac{ c^2}{(n+1)^{c}}\Big(1+ \frac{1}{n_0}\Big)^{c}\varphi_{c-1}(n).
 \notag\\
\end{alignat}
Therefore,  for $\alpha =1$, we obtain,
\begin{equation}
 s_{n+1} \leq s_{n_0}\Big(\frac{n_0}{n+1}\Big)^{c}+
 \frac{\tau c^2}{(n+1)^{c}}\Big(1+ \frac{1}{n_0}\Big)^{c} \varphi_{c-1}(n),
\end{equation}
which completes the proof.
\end{proof}

\end{document}